\documentclass[10pt,reqno]{amsart}
\usepackage[a4paper,top=3 cm,bottom=3 cm,left=2.5 cm,right=2.5 cm]{geometry}

\usepackage[utf8]{inputenc}  
\usepackage[T1]{fontenc}
\usepackage{pgfplots}
\usepackage{subfig}
\usepackage{helvet}
\usepackage{graphicx}
\usepackage{xcolor}
\usepackage{enumitem} 
\usepackage{mathrsfs} 
\usepackage{amsmath,amssymb, mathtools}
\usepackage{hyperref}
\usepackage{cleveref}
\usepackage{dsfont}

\newcommand{\wconv}{\rightharpoonup}
\newcommand{\wconvs}{\overset{*}{\rightharpoonup}}
\DeclareMathOperator{\diver}{div}

\def\R{\mathbb R}
\def\N{\mathbb N}

\def\to{\rightarrow}

\newtheorem{theorem}{Theorem}[section]

\newtheorem{corollary}{Corollary}[section]
\newtheorem{lemma}{Lemma}[section]

\newtheorem{remark}{Remark}[section]

\newtheorem{proposition}{Proposition}[section]

\usepackage{titletoc}

\makeatletter
\newcommand\thankssymb[1]{\textsuperscript{\@fnsymbol{#1}}}

\makeatother

\begin{document}
                
\title[Uniform Turnpike Property and Singular Limits]{Uniform Turnpike Property and Singular Limits}

\author[M. Hernández]{Martín Hernández\thankssymb{2}}
\email{martin.hernandez@fau.de}

\author[E. Zuazua]{Enrique Zuazua \thankssymb{1}\thankssymb{2}\thankssymb{3}}

\thanks{\thankssymb{2}
Chair for Dynamics, Control, Machine Learning, and Numerics, Alexander von Humboldt-Professorship, Department of Mathematics,  Friedrich-Alexander-Universit\"at Erlangen-N\"urnberg,
91058 Erlangen, Germany.}

\thanks{\thankssymb{1} 
 Departamento de Matem\'{a}ticas,
Universidad Aut\'{o}noma de Madrid,
28049 Madrid, Spain.
}

\thanks{\thankssymb{3} 
Chair of Computational Mathematics, Fundaci\'{o}n Deusto. Av. de las Universidades, 24,
48007 Bilbao, Basque Country, Spain.
}
\email{\texttt{enrique.zuazua@fau.de}}



\begin{abstract}

Motivated by singular limits for long-time optimal control problems, we investigate a class of parameter-dependent parabolic equations. First, we prove a turnpike result, uniform with respect to the parameters within a suitable regularity class and under appropriate bounds. The main ingredient of our proof is the justification of the uniform exponential stabilization of the corresponding Riccati equations, which is derived from the uniform null control properties of the model. 

Then, we focus on a heat equation with rapidly oscillating coefficients. In the one-dimensional setting, we obtain a uniform turnpike property with respect to the highly oscillatory heterogeneous medium. Afterward, we establish the homogenization of the turnpike property.  
Finally, our results are validated by numerical experiments.
\end{abstract}

\maketitle
\begin{center}
    \emph{Dedicated to Shi Jin on his 60th birthday, with friendship and admiration.}
\end{center}

\tableofcontents


\section{Introduction}

In the context of long-time horizon optimal control, the turnpike property ensures that optimal controls and solutions remain close to the optimal solution of the corresponding stationary optimal control problem most of the time. This stationary path is called the turnpike, which refers to the fastest route linking points that are far away enough. The turnpike phenomenon has been extensively studied for different equations in recent years (see \cite{longtime,MHSZRL,MR3780737,TURNPIKE_ONDA,z_turnpike,WZ3,no_lineal_turnpike,Integral_and_measure,remarks,Sensitivity_Analysis,Exponential_sensitivity_1,Esteve_Yag_e_2022}, and the references therein, for example). An extended and comprehensive survey can be found in \cite{geshkovski_zuazua_2022}.  In particular, it is well understood that the turnpike property relies on two ingredients: firstly, the cost functional must penalize both the state and control; secondly, the system must be controllable or stabilizable.

Building on the work of Porretta and Zuazua \cite{longtime}, we delve into the turnpike property within the context of parameter-dependent parabolic optimal control problems. We prove that when the null controllability holds, uniformly with respect to the parameters, the turnpike property is uniform as well.

Previous analysis of optimal control problems governed by parameter-dependent partial differential equations can be found in \cite{MR3076074} and the references included therein. The connections between parameterized control problems and the turnpike property have also been analyzed in \cite{MR3905433}, where greedy approximate algorithms for parameterized control problems have been developed. See also \cite{lazar2022greedy}.

This article complements the existing literature by showing that the turnpike property is uniform and enjoys homogenization properties in a suitable context.

Throughout this article, we primarily concentrate on parabolic equations to streamline the presentation. Nevertheless, our methodology can be extended to other models, such as finite-dimensional systems or wave-type equations, provided that the uniform null control property, with respect to the relevant parameters,  is satisfied.

\subsection{Problem formulation}
Let $\Omega$ be a bounded Lipschitz domain in $\R^n$, $n\geq 1$, and the time horizon $T>0$. We define the following optimal control problem
\begin{align}\label{evolutive_control_problem}
    \min_{ f\in L^{2}(0,T;\Omega)}\biggr\{ J^T( f)=\frac{1}{2}\int_0^T\left( \| f(\cdot,t)\|_{L^2(\Omega)}^2 +\| y(\cdot,t)-y_d(\cdot)\|^2_{L^2(\Omega)}\right)dt\biggr\},
\end{align}
where $ y$ is the solution of the parabolic equation
\begin{align}\label{rapidly_ocilation_general_heat}
   \begin{cases}
     y_t -\diver( a(x) \nabla   y)+b(x) \cdot \nabla   y + p(x) y=\chi_\omega  f\quad &(x,t)\in\Omega\times (0,T),\\
      y(x,t)=0 &(x,t)\in\partial\Omega\times (0,T),\\
      y(x,0)=y_0&x\in\Omega.
    \end{cases}
\end{align}
Here, $y_d\in L^2(\Omega)$ is a time-independent target, $ y$ is the state, $ f\in L^{2}(0,T;\Omega)$ is the control, and $y_0\in L^2(\Omega)$ is the initial condition. The open set $\omega\subset\Omega$ is nonempty, and $\chi_\omega$ denotes the characteristic function of the set $\omega$ where the control is being applied. We denote by $( y, f)$ the optimal time-dependent pair of \eqref{evolutive_control_problem}, which, of course, depends also in the length $T$ of the time-horizon.

We assume that the coefficients $(a,b,p)$ are bounded in the class
\begin{align}\label{class_definition}
    \mathscr{C}:=W^{1,\infty}(\Omega)\times \left(L^\infty(\Omega)\right)^n\times L^\infty(\Omega),
\end{align}
and satisfy the uniform ellipticity condition in the principal part, i.e.
\begin{align}\label{elipticity_of_a}
    0<a_0\leq a(x) \quad \text{a.e. in }\Omega. 
\end{align}
Then, by the classical global Carleman inequalities introduced by Fursikov and Imanuvilov in \cite{MR1406566}, we can guarantee the uniform null controllability of the system \eqref{rapidly_ocilation_general_heat}, that is, for any $T>0$ and $y_0\in L^2(\Omega)$, there exists  $ f\in L^2(0,T;\Omega)$ such that we can drive to zero the solution of \eqref{rapidly_ocilation_general_heat} in time $T$, the control $f$ being bounded in $L^2(0,T;\Omega)$ by a constant independent of the coefficients within the considered class (see Theorem \ref{teo_controlabilidad_uniforme_0} below).

Let us now consider the stationary optimal control problem
\begin{align}\label{stationary_var_problem}
    \min_{ \overline{f} \in  L^2(\Omega)}\biggr\{J^s( \overline{f})= \frac{1}{2}\biggr(\| \overline{f}(\cdot)\|_{L^2(\Omega)}^2+\|\overline{y}(\cdot)-y_d(\cdot)\|^{2}_{L^2(\Omega)}\biggr)\biggr\}.
\end{align}
Here $\overline{y}$ solves the elliptic equation
\begin{align}\label{stationry_system}
   \begin{cases}
        -\diver( a(x) \nabla  \overline{y})+b(x) \cdot\nabla  \overline{y} + p(x)\overline{y}=\chi_\omega  \overline{f}\quad & x\in\Omega,\\
       \overline{y}(x)=0 &x\in\partial\Omega,
   \end{cases}
\end{align}
where $y_d\in L^2(\Omega)$ is the same target of the evolution control problem \eqref{evolutive_control_problem}, and the coefficients are the same as in \eqref{rapidly_ocilation_general_heat}.

To avoid additional technical difficulties and ensure the existence and uniqueness of \eqref{stationry_system}, we assume that the coefficients $a$, $b$, and $p$ are such that
\begin{align}\label{condition_kernel}
   \operatorname{Ker}(\mathcal{A}^*) = \{0\},
\end{align}
where $\mathcal{A}^*$ is the adjoint operator of
\begin{align}\label{operator_A}
    \mathcal{A}= -\diver( a(x) \nabla  \,\cdot\,)+b(x) \nabla  \,\cdot\, + p(x)\,\cdot\,,
\end{align}
that is, 
\begin{align}\label{operator_A_adjoint}
\mathcal{A}^*= -\diver( a(x) \nabla\,\cdot\,)-\diver(b(x)\,\cdot\,)+p(x)\,\cdot\,.
\end{align}
Denote by $(\overline{y}, \overline{f})$ the optimal pair of \eqref{stationary_var_problem}. Although hypothesis \eqref{condition_kernel} does not guarantee that the solution of \eqref{stationry_system} is uniformly bounded in $H^1(\Omega)$ for a fixed right-hand side term, by assuming condition \eqref{elipticity_of_a} and that the coefficients $(a,\, b,\,p)$ are bounded in $\mathscr{C}$, then extra bounds that the coercivity of $J^s$ yield, ensure that the optimal states $\overline{y}$ are uniformly bounded in $H^1(\Omega)$, see Lemma \ref{lemma_1}. This is an essential aspect of the uniform turnpike property that requires all static optimal states and controls to be uniformly bounded.

Our first main result guarantees the exponential uniform turnpike property in this setting. Namely, we have the following theorem. 

\begin{theorem}[Uniform turnpike property for coefficients bounded in $\mathscr{C}$]\label{TH_TURNPIKE_EXP}   
   Let us assume that the coefficients $(a,\,b,\,p)$ are bounded in $\mathscr{C}$ and that \eqref{elipticity_of_a} and \eqref{condition_kernel} are fulfilled. Consider the time-dependent optimal pairs $( y, f)$ and the static ones $(\overline{y}, \overline{f})$, of problems \eqref{evolutive_control_problem} and \eqref{stationary_var_problem}, respectively. Then, there exist two positive constants $C$ and $\mu$ independent of $T\geq 1$ and the coefficients $(a,\,b,\,p)$ in this class such that
\begin{align}\label{ineq_exponential_turnpike}
    \| y(\cdot,t)-\overline{y}(\cdot)\|_{L^2(\Omega)}+\| f(\cdot,t)- \overline{f}(\cdot)\|_{L^2(\Omega)}\leq C\left(\|y_0\|_{L^2(\Omega)}+\|y_d\|_{L^2(\Omega)}\right)\biggr(e^{-\mu t}+e^{-\mu(T-t)}\biggr),
\end{align}
for every $t\in(0,T)$, $y_0\in L^2(\Omega)$, and coefficients $(a,\,b,\,p)$  in this class.

\end{theorem}
To the best of our knowledge, this is the first result in this direction in the literature.

In the previous theorem, the constants $C$ and $\mu$ depend, in particular, on the ellipticity constant $a_0$ in \eqref{elipticity_of_a}, and the uniform bound on the coefficients in $\mathscr{C}$, because the uniform controllability constant depends on these conditions and bounds as well.

This result not only ensures the turnpike property for  \eqref{evolutive_control_problem} but also guarantees that for a family of coefficients, $(a,\, b,\, p)$, uniformly bounded in $\mathscr{C}$ and satisfying \eqref{elipticity_of_a} and \eqref{condition_kernel}, there exists a uniform tubular neighborhood (defined by the turnpike constants \eqref{ineq_exponential_turnpike}) of the steady optimal configurations constraining the optimal dynamic trajectories (states and controls). This assertion could not be concluded in the case where the family $(a,\, b,\, p)$ is not uniformly bounded in $\mathscr{C}$, because of the lack of uniform controllability.

The proof of this theorem uses, in an essential manner, the uniform controllability of the system and it is based on the decoupling strategy in \cite{longtime}. This is done through the use of the Riccati operator of the associated infinite-time horizon problem. The uniform null controllability property ensures the uniform exponential stability of the Riccati operator (Proposition \ref{corollary_1_exponencial}). Once this fact is proved, the result follows similarly as in \cite{longtime}. 

Theorem \ref{TH_TURNPIKE_EXP} does not apply to highly oscillatory heterogeneous media in homogenization theory, since the corresponding coefficients are not uniformly bounded in $\mathscr{C}$, and Carleman inequalities do not guarantee the uniform controllability in this context. However, according to Alessandrini and Escauriaza \cite{COCV_2008__14_2_284_0}, in one-space dimension $n=1$, the uniform null control property holds by simply assuming coefficients to be uniformly bounded in $\left[L^\infty(\Omega)\right]^3$ and uniformly elliptic \eqref{elipticity_of_a}. This uniform controllability property was previously proved in \cite{MR1922469} for periodic homogenization in one-space dimension. According to these uniform controllability results, we have the following result.


\begin{theorem}[One-dimensional uniform turnpike property for coefficients bounded in  ${\left[L^\infty(\Omega)\right]^3}$]\label{one_dimensional_UT}
Consider $\Omega=(0,1)$ and assume that the coefficients $(a,b,p)$ are bounded in $\left[L^\infty(\Omega)\right]^3$, and that conditions \eqref{elipticity_of_a} and \eqref{condition_kernel} are fulfilled. Under these milder restrictions on the coefficients, there exist two positive constants $C$ and $\mu$, independent of $T\geq 1$ and the coefficients $(a,\,b,\,p)$ in this class, such that \eqref{ineq_exponential_turnpike} holds. Furthermore, in the particular case of periodic rapidly oscillatory coefficients, both the uniform turnpike property and the homogenization of the turnpike hold.
\end{theorem}

\subsection{Outline} The rest of this work is organized in the following way. In Section \ref{sec:preliminaries}, we present preliminary results related to well-posedness and uniform null controllability. Section \ref{main_results} is devoted to state some consequences of the uniform turnpike property, including the uniform integral turnpike property and the homogenization of the turnpike property in $1-d$, in the context of highly oscillatory heterogeneous media. The proof of the uniform turnpike property is given in Section \ref{Section_main_theo}. In Section \ref{numerical}, we present numerical experiments that confirm our theoretical results. Finally, Section \ref{further_comentaries} concludes the paper with a discussion on possible extensions and open problems.

\section{Preliminaries}\label{sec:preliminaries}
Let us consider the coefficients $(a,b,p)\in L^{\infty}(\Omega)\times \left(L^\infty(\Omega)\right)^n\times L^\infty(\Omega)$ satisfying \eqref{elipticity_of_a}. Then, for every $y_0\in  L^2(\Omega)$, the parabolic equation \eqref{rapidly_ocilation_general_heat} admits a unique solution $ y$ in the class 
\begin{align*}
    W(0,T)=\left\{  y \in L^2(0,T;H^1_0(\Omega)),\,   y_t\in L^2(0,T;H^{-1}(\Omega))\right\}.
\end{align*}
See, for instance, \cite[Chapter 7]{MR1625845}.
Additionally, if the coefficients $(a,\,b,\,p) \in L^\infty(\Omega)\times\left(L^\infty(\Omega)\right)^n\times L^\infty(\Omega)$ satisfy \eqref{elipticity_of_a} and \eqref{condition_kernel}, then  the elliptic equation \eqref{stationary_var_problem} admits a unique solution $\overline{y}$ in $H_0^1(\Omega)$. See \cite[Chapter 6]{MR1625845}.

The following lemma shows that problems \eqref{evolutive_control_problem} and \eqref{stationary_var_problem} admit a unique minimizer.

\begin{lemma}[Existence and uniqueness of the optimal pairs]\label{wellposedness}
Let $(a,b,p)\in L^{\infty}(\Omega)\times \left(L^\infty(\Omega)\right)^n\times L^\infty(\Omega)$. Under the assumption \eqref{elipticity_of_a} the optimization problem \eqref{evolutive_control_problem} has a unique solution $( y, f )\in W(0,T)\times L^2(0,T;\Omega)$, where $ y$ is the optimal state associated to
the control $ f$. Similarly, assuming \eqref{condition_kernel} is fulfilled, the stationary optimization problem \eqref{stationary_var_problem} has a unique solution $(\overline{y}, \overline{f} )\in H_0^1(\Omega)\times L^2(\Omega)$, where $\overline{y}$ is the optimal state associated to
the control $ \overline{f}$.
\end{lemma}
The proof of Lemma \ref{wellposedness} is standard and is based on the direct method in the calculus of variations. We omit it for brevity.

\begin{remark}
Assumption \eqref{condition_kernel} is satisfied, in particular, if $(a,\, b,\, p) \in L^{\infty}(\Omega) \times \left(L^\infty(\Omega)\right)^n \times L^\infty(\Omega)$, satisfy \eqref{elipticity_of_a} and  $p(x) - \diver q(x)/2\geq 0$ a.e. in $\Omega$. See \cite[Chapter 9, Remark 23]{MR2759829} or \cite[Corollary 8.2]{MR1814364}.
\end{remark}

\subsection{Uniform null controllability}\label{uniform_null_section}
 In this section, we analyze the uniform null controllability properties of the parabolic equation \eqref{rapidly_ocilation_general_heat}. The following holds:

\begin{theorem}[Uniform null controllability in the multi-dimensional setting]\label{teo_controlabilidad_uniforme_0}
Let us assume that the coefficients $(a,\,b,\,p)$ are bounded in $\mathscr{C}$ and satisfy the uniform ellipticity condition \eqref{elipticity_of_a}. Then, the system \eqref{rapidly_ocilation_general_heat} satisfies the uniform null controllability property, in the sense that for each  $T>0$, there exists a constant $C_T>0$ independent of the coefficients $(a,b,p)\in \mathscr{C}$ such that for each $y_0\in L^2(\Omega)$ there exists a null control $f$ satisfying
\begin{align}\label{uniform_bound_control}
    \| f\|_{L^2(0,T;\Omega)}\leq C_{T}\|y_0\|_{L^2(\Omega)}, \quad \forall y_0\in L^2(\Omega).
\end{align}
\end{theorem}
The constant $C_{T}$ in Theorem \ref{teo_controlabilidad_uniforme_0} is the so-called \emph{controllability cost} and is uniformly bounded in this class of coefficients.
\begin{proof}
Due to the global Carleman inequalities introduced in \cite{MR1406566}, there exists a control $f\in L^2(0,T;\Omega)$ such that $y$, the solution of \eqref{rapidly_ocilation_general_heat} with initial condition $y_0\in L^2(\Omega)$ and coefficients $(a,\,b,\,p)\in\mathscr{C}$, satisfies 
\begin{align*}
y(\cdot,T)=0,\quad  \text{ in }\Omega,
\end{align*}
and there exists is a positive constant $C_{a,b,p,T}$ such that
\begin{align*}
\| f\|_{L^2(0,T;\Omega)}\leq C_{a,b,p,T}\|y_0\|_{L^2(\Omega)}, \quad\forall y_0\in L^2(\Omega).
\end{align*}
Also, from \cite{MR1406566}, we can observe that the controllability cost $C_{a,b,p, T,a_0}$ depends continuously on the norm of the coefficients $(a,\,b,\,p)$ in $\mathscr{C}$, the time horizon $T$, and the ellipticity constant $a_0$. Thus, there exists a positive constant $C_{T,a_0}$ such that $ C_{a,b,p,T}\leq C_{T,a_0}$ for all coefficients in this class, concluding \eqref{uniform_bound_control}.
\end{proof}

Theorem \ref{teo_controlabilidad_uniforme_0} still holds when we assume $a\in L^{\infty}(\Omega)$, but only in one-space dimension $n=1$. More precisely, using complex analysis tools, such as K-quasiconformal homeomorphisms, Alessandrini and Escauriaza proved the following result in \cite{COCV_2008__14_2_284_0}.

\begin{theorem}[Uniform null controllability in $1-d$]\label{theorema_1d_null_control}
Let $\Omega=(0,1)$ and assume that the coefficients $(a,\,b,\,p)$ are bounded in $L^{\infty}(\Omega)\times \left(L^\infty(\Omega)\right)^n\times L^\infty(\Omega)$ and satisfy the uniform ellipticity condition \eqref{elipticity_of_a}.
Then, the system \eqref{rapidly_ocilation_general_heat} satisfies the uniform null controllability property.
\end{theorem}

This result was previously proved in \cite{MR1922469}, in the context of periodic homogenization in $1-d$ with boundary control.

\section{Some consequences of the uniform turnpike property}\label{main_results}
In this section, we present some results that arise as consequences of the uniform turnpike property.

\subsection{Uniform integral turnpike property}

We begin with the so-called integral turnpike property. As a direct consequence of the uniform turnpike property, we can ensure that the integral turnpike property holds uniformly.
\begin{corollary}[Uniform integral turnpike property]\label{Theo_integral_turnpike}
 Let us assume that the coefficients $(a,\,b,\,p)$ are bounded in $\mathscr{C}$. Additionally, assume that \eqref{elipticity_of_a} and \eqref{condition_kernel} are fulfilled. Consider the optimal pairs $( y, f)$ and $(\overline{y}, \overline{f})$ of problems \eqref{evolutive_control_problem} and \eqref{stationary_var_problem}, respectively. Then for $T>1$, we have
\begin{align}
    \nonumber\left\|\frac{1}{T}\int_0^T  y(\cdot,t)dt - \overline{y}(\cdot)\right\|_{L^2(\Omega)}+ &\left\| \frac{1}{T}\int_0^T  f(\cdot,t)dt -  \overline{f}(\cdot) \right\|_{L^2(\Omega)}\\
    &\leq \frac{2C(\|y_0\|_{L^2(\Omega)}+\|y_d\|_{L^2(\Omega)})(1-e^{-\mu T})}{\mu T},
\end{align}
for every $y_0\in L^2(\Omega)$ and  $(a,\,b,\,p)$ in this class. Here $C$ and $\mu$ are the same constants as in Theorem \ref{TH_TURNPIKE_EXP}.
\end{corollary}
This result can be derived immediately by integrating \eqref{ineq_exponential_turnpike} over time; hence, we will omit the proof.

\subsection{Application to $1-d$ homogenization}
Let $a\in L^\infty(\R)$ be a periodic function with period $1$ satisfying
\begin{align}\label{bound_a_homo}
    0<a_0\leq a(x)\leq a_1 \text{ a.e. in }\R,
\end{align}
and consider 
\begin{align}\label{homogenizated_a}
    a_\varepsilon(x)=a\left(\frac{x}{\varepsilon}\right),\quad \text{for }\varepsilon>0.
\end{align}
We are interested in analyzing the uniform turnpike property for the optimal control problem
\begin{align}\label{evolutive_control_problem_2}
    \min_{f^\varepsilon\in L^{2}(0,T;\Omega)}\biggr\{ J^T(f^\varepsilon)=\frac{1}{2}\int_0^T \|f^\varepsilon(\cdot,t)\|_{L^2(\Omega)}^2 +\|y^\varepsilon(\cdot,t)-y_d(\cdot)\|^2_{L^2(\Omega)}dt\biggr\},
\end{align}
where $ y^\varepsilon$ solves the rapidly oscillating heat equation
\begin{align}\label{rapidly_heat}
    \begin{cases}
        y_t^\varepsilon -  \left(  a_\varepsilon(x)y^\varepsilon_x\right)_x=\chi_\omega f^\varepsilon\quad & (x,t)\in(0,1)\times(0,T),\\
    y^\varepsilon(0,t)=y^\varepsilon(1,t)=0 & t\in(0,T),\\
    y^\varepsilon(x,0)=y_0(x)& x\in(0,1),
    \end{cases}
\end{align}
for $\varepsilon>0$. We also consider the stationary problem
\begin{align}\label{stationary_var_problem_2}
    \min_{ \overline{f}^\varepsilon \in  L^2(\Omega)}\biggr\{J^s( \overline{f}^\varepsilon)= \frac{1}{2}\biggr(\| \overline{f}^\varepsilon(\cdot)\|_{L^2(\Omega)}^2+\|\overline{y}^\varepsilon(\cdot)-y_d(\cdot)\|^{2}_{L^2(\Omega)}\biggr)\biggr\},
\end{align}
where $\overline{y}^\varepsilon$ solves the rapidly oscillating elliptic equation
\begin{align}\label{stationry_system_2}
   \begin{cases}
        - \left(  a_\varepsilon(x) \overline{y}^\varepsilon_x\right)_x=\chi_\omega  \overline{f}^\varepsilon\quad & x\in(0,1),\\
        \overline{y}^\varepsilon(0)=\overline{y}^\varepsilon(1)=0.
   \end{cases}
\end{align}
In this setting, thanks to Theorems \ref{theorema_1d_null_control} and \ref{one_dimensional_UT}, we can immediately deduce the uniform turnpike property.

\begin{corollary}[Uniform turnpike property and homogenization in $1-d$]\label{corollary_uniform_epsilon}
    Let $a_\varepsilon\in L^{\infty}(\R)$ be as in \eqref{bound_a_homo}-\eqref{homogenizated_a}.  Consider the optimal pairs $( y^\varepsilon, f^\varepsilon)$ and $(\overline{y}^\varepsilon, \overline{f}^\varepsilon)$ of problems \eqref{evolutive_control_problem_2} and \eqref{stationary_var_problem_2}, respectively. Then for $T>1$, there exist two positive constants $C$ and $\mu$, independent of $T$ and $\varepsilon$, such that
\begin{align}\label{ineq_exponential_turnpike_varepsilon}
    \| y^\varepsilon(\cdot,t)-\overline{y}^\varepsilon(\cdot)\|_{L^2(0,1)}+\| f^\varepsilon(\cdot,t)- \overline{f}^\varepsilon(\cdot)\|_{L^2(0,1)}\leq C\left(\|y_0\|_{L^2(0,1)}+\|y_d\|_{L^2(0,1)}\right)\biggr(e^{-\mu t}+e^{-\mu(T-t)}\biggr),
\end{align}
for every $t\in (0,T)$, $\varepsilon>0$, and $y_0\in L^2(\Omega)$.
\end{corollary}

Denote by $a_h$ the homogenized effective constant given by
\begin{align}\label{eq_ah}
    a_h= \left(\int_{0}^1 \frac{1}{a(x)}dx\right)^{-1} .
\end{align}
Then, if $f^\varepsilon\to f^h$ as $\varepsilon\to0$ in $L^2(0,T;\Omega)$, classical results in homogenization theory (see Appendix \ref{homogenization_theorem_sec2_proof}) guarantee that the solution $y^\varepsilon$ of \eqref{rapidly_heat} converges in $C([0,T];L^2(\Omega))$ to the solution of the homogenized heat equation
\begin{align}\label{lim_system}
    \begin{cases}
    y_t^h -a_h\, y^h_{xx}=f^h\quad & (x,t)\in(0,1)\times(0,T),\\
    y^h(0,t)=y^h(1,t)=0 & t\in(0,T),\\
    y^h(x)=y_0(x)& x\in(0,1).
\end{cases}
\end{align}
Similarly, if $\overline{f}^\varepsilon \to\overline{f}^h$ as $\varepsilon\to 0$  weakly in $L^2(\Omega)$, the solution of  \eqref{stationry_system_2} converges weakly in $H_0^1(\Omega)$ to the solution of the homogenized elliptic equation
\begin{align}\label{lim_system_stationary}
    \begin{cases}
     -a_h\, \overline{y}_{xx}^h=\overline{f}^h\quad  x\in(0,1),\\ \overline{y}^h(0)=\overline{y}^h(1)=0.
   \end{cases}
\end{align}
Since Corollary \ref{corollary_uniform_epsilon} is uniform with respect to $\varepsilon>0$, the homogenization of the turnpike property holds. In other words, we can take the limit as $\varepsilon$ goes to zero in \eqref{ineq_exponential_turnpike_varepsilon} and deduce the following result, already stated in \cite{longtime}.

\begin{corollary}\label{turnpike_limit_states}
Let us denote by $(y,f)$ and $(\overline{y},\overline{f})$ the optimal pairs of the optimization problems \eqref{evolutive_control_problem_2} and  \eqref{stationary_var_problem_2} subject to the homogenized equations \eqref{lim_system} and \eqref{lim_system_stationary}, respectively. Then, for $T\geq 1$ we have
\begin{align}\label{homogenization_of_the_turnpike}
    \|y(\cdot,t)-\overline{y}(\cdot)\|_{L^2(0,1)}+\|f(\cdot,t)-\overline{f}(\cdot)\|_{L^2(0,1)}\leq C\left(\|y_0\|_{L^2(0,1)}+\|y_d\|_{L^2(0,1)}\right)\biggr(e^{-\mu t}+e^{-\mu(T-t)}\biggr),
\end{align}
for every $t\in(0,T)$, where the positive constants $C$ and $\mu$ are those in \eqref{ineq_exponential_turnpike_varepsilon}.
\end{corollary}
The proof of the strong convergence of the optimal pairs for both the evolutionary and stationary problems can be found in Proposition \ref{homogenization_theorem_sec2}.

\section{Proof of the uniform turnpike property}\label{Section_main_theo}
We devote this section to the proof of the uniform turnpike property (Theorem \ref{TH_TURNPIKE_EXP}).

In what follows, $C$ will denote a positive constant that may change from line to line and can depend on $a_0$. However, it will always be independent of the coefficients $(a, b, p)$ and the time horizon $T$. We also denote by $(\cdot,\cdot)_{L^2(\Omega)}$ the inner product in $L^2(\Omega)$.

Let us recall the operator $\mathcal{A}$ defined in \eqref{operator_A}, and its adjoint operator defined in \eqref{operator_A_adjoint}.
Thus, let us introduce the adjoint states which characterize the optimal controls of problems \eqref{evolutive_control_problem} and \eqref{stationary_var_problem}.

\begin{lemma}\label{optimality_system}Assume the coefficients $a,\,b$ and $p$ are in $\mathscr{C}$. Additionally, assume that \eqref{elipticity_of_a} and \eqref{condition_kernel} are fulfilled. Consider the optimal controls $ f$ and $ \overline{f}$ of problems \eqref{evolutive_control_problem} and \eqref{stationary_var_problem}, respectively. Then, the optimal control $ f$ can be characterized by the identity $ f=-\chi_\omega\psi$, where $\psi$ satisfies
\begin{align}\label{adj_var_op}
 \begin{cases}
    -\psi_t +\mathcal{A}^*\psi= y - y_d\quad & (x,t)\in\Omega\times(0,T),\\
    \psi(x,t)=0 &(x,t)\in\partial\Omega\times(0,T),\\
    \psi(x,T)=0&x\in \Omega,
    \end{cases}
\end{align}
with $ y$ the optimal state associated to $ f$. 

Similarly, the optimal control $ \overline{f}$ is characterized by $ \overline{f}=-\chi_\omega\overline{\psi}$ with  $\overline{\psi}$ solution of
\begin{align}\label{stationry_adj_system}
   \begin{cases}
   \mathcal{A}^*\overline{\psi}=\overline{y}-y_d\quad & x\in\Omega,\\
       \overline{\psi}(x)=0&x\in\partial \Omega,
   \end{cases}
\end{align}
where $\overline{y}$ is the optimal state associated to $ \overline{f}$.
\end{lemma}

The proof of Lemma \ref{optimality_system} is standard and can be found in the classical book of Lions \cite[Theorem 1.4, Chapter II, and Theorem 2.1, Chapter III]{lions_book}, for example.

The proof of the following lemma, which establishes some energy estimates, can be found in Appendix \ref{lemma_1_proof}.

\begin{lemma}\label{lemma_1}
Assume the coefficients $a,\,b$ and $p$ are bounded in $\mathscr{C}$. Additionally, assume that \eqref{elipticity_of_a} and \eqref{condition_kernel} are fulfilled. Let $( y, f,\psi)$ and $(\overline{y}, \overline{f},\overline{\psi})$ be the optimal state-control-adjoint triples of problems \eqref{evolutive_control_problem} and \eqref{stationary_var_problem}, respectively. Then, there exists a constant $C>0$ independent of $T\geq 1$ and the coefficients such that
\begin{align*}
    \| y(\cdot,T)\|^{2}_{L^{2}(\Omega)}\leq  C\biggr(\int_{0}^T \| f(\cdot,t)\|^2_{L^{2}(\omega)} +\| y(\cdot,t)\|^2_{L^2(\Omega)}dt\biggr) +\|y_0\|^2_{L^2(\Omega)},
\end{align*}
and
\begin{align*}
    \|\psi(\cdot,0)\|^{2}_{L^{2}(\Omega)}\leq C\biggr(\int_{0}^T\|\psi(\cdot,t)\|^{2}_{L^2(\omega)}+\| y(\cdot,t)-y_d\|^2_{L^2(\Omega)}dt\biggr).
\end{align*}
Furthermore, 
\begin{align*}
\|\overline{y}(\cdot)\|^{2}_{H^1(\Omega)}+\|\overline{\psi}(\cdot)\|_{H^1(\Omega)}^2+\| \overline{f}(\cdot)\|^2_{L^2(\Omega)}  \leq \hat{C}\|y_d(\cdot)\|^{2}_{L^{2}(\Omega)},
\end{align*}
with $\hat{C}>0$ independent of the coefficients $(a,\,b,\,p)\in \mathscr{C}$.
\end{lemma}
The second main ingredient of the proof of Theorem \ref{TH_TURNPIKE_EXP} is the Riccati equation that characterizes the optimal control problem and analyzes its behavior when $T\to\infty$.

Let us introduce the \emph{time-dependent Riccati operator}. We start by analyzing the problem \eqref{evolutive_control_problem} with $y_d\equiv0$, i.e. the optimal control problem 
\begin{align}\label{problem_J0}
    \min_{ f\in L^{2}(0,T;\Omega)}\biggr\{ J^{T,0}( f)=\frac{1}{2}\int_0^T \| f(\cdot,t)\|_{L^2(\Omega)}^2  +\| y(\cdot,t)\|^2_{L^2(\Omega)}dt\biggr\},
\end{align}
where $ y$ solves \eqref{rapidly_ocilation_general_heat}. Let us define the operator $\mathcal{E}(T): L^2(\Omega) \to L^2(\Omega)$ such that
\begin{align*}
    \mathcal{E}(T)y_0(x) :=\psi(x,0).
\end{align*}
The operator $\mathcal{E}(T)$ depends on the coefficients $(a, b, p)$. However, we will not make this dependence explicit  to keep the notation simple.

The next lemma summarizes some useful properties of the operator $\mathcal{E}(T)$. See Appendix \ref{Lemma_exponencial_proof} for the proof.

\begin{lemma}\label{Lemma_exponencial}
Assume the coefficients $a,\,b$ and $p$ are bounded in $\mathscr{C}$. Additionally, assume that \eqref{elipticity_of_a} and \eqref{condition_kernel} are fulfilled. Then the operator $\mathcal{E}(T)$ is well-defined, linear, and continuous from $L^2(\Omega)$ to $L^2(\Omega)$. It also satisfies that
\begin{enumerate}
    \item There exists a constant $C>0$, independent of $T$ and the coefficients $(a,b,p)$, such that
\begin{align*}
(\mathcal{E}(T)y_0(\cdot),y_0(\cdot))_{L^2(\Omega)}=  \min_{ f\in L^{2}(0,T;\Omega)}J^{T,0}( f)\leq C\|y_0(\cdot)\|^2_{L^2(\Omega)}.
\end{align*}
\item The limit $\lim_{T\to\infty}(\mathcal{E}(T)y_0,y_0)$ is finite.
\end{enumerate}
\end{lemma}
We note that for each $t\in(0,T)$
\begin{align*}
  \mathcal{E}(T-t) y(x,t)=\psi(x,t).
\end{align*}
This is a consequence of the uniqueness of the state-adjoint pairs $(y,\psi)$ solution of \eqref{rapidly_ocilation_general_heat} and \eqref{adj_var_op}, respectively. For more details, see \cite[Chapter 3, Section 4]{lions_book} or \cite[Chapter 4]{MR4211767}.

Consider the following infinite-time horizon control problem
\begin{align}\label{problem_Jinf}
    \min_{ \hat{f}\in L^{2}(0,\infty;\Omega)}\biggr\{ J^{\infty,0}( \hat{f})=\frac{1}{2}\int_0^\infty \| \hat{f}(\cdot,t)\|_{L^2(\Omega)}^2  +\| \hat{y}(\cdot,t)\|^2_{L^2(\Omega)}dt\biggr\},
\end{align}
where $ \hat{y}$ solves \eqref{rapidly_ocilation_general_heat} with time horizon $T=\infty$. Denote by $( \hat{y}, \hat{f},\hat{\psi})$ its optimal variables, and define the operator $ \hat{E}(0): L^2(\Omega) \to L^2(\Omega)$ by setting 
\begin{align*}
    \hat{E}(0)y_0(x) :=\hat{\psi}(x,0).
\end{align*}
Similarly to the operator $\mathcal{E}(T)$, we can define $\hat{E}(s) \hat{y}(s,x)=\hat{\psi}(s,x)$ for $s\in(0,\infty)$. However, it is possible to show that the operator $\hat{E}(s)$ is independent of $s$ (i.e., $\hat{E}(s)=\hat{E}$) and is thus referred to as the \emph{stationary Riccati operator} (see \cite[Chapter 3, Section 4]{lions_book}). Although the operator $\hat{E}$ also depends on the coefficients, we do not explicitly state this dependence either.

The following lemma relates the asymptotic behavior of the time-dependent Riccati operator $\mathcal{E}$ to the stationary Riccati operator $\hat{E}$.

\begin{lemma}\label{Lemma_3_exponencial} Assume the coefficients $a,\,b$ and $p$ are bounded in $\mathscr{C}$. Additionally, assume that \eqref{elipticity_of_a} and \eqref{condition_kernel} are fulfilled. The optimal control problem \eqref{problem_Jinf} and the Riccati operator $\hat{E}\in \mathcal{L}(L^2(\Omega))$ are well-defined.
For the optimal control $ f$ of problem \eqref{problem_J0}, we have the feedback backward characterization $ f(x,t)=-\chi_\omega\mathcal{E}(T-s) y(x,s)$. Moreover, we have
\begin{align*}
    \mathcal{E}(t)y_0\longrightarrow\hat{E}y_0\quad\text{strongly in }L^2(\Omega), \text{ when } t\to\infty.
\end{align*}
 Finally, the optimal control of problem \eqref{problem_Jinf} can be characterized by the identity $ \hat{f}(x,t)=-\chi_\omega \hat{E} \hat{y}(x,t)$.
\end{lemma}
The proof of this lemma can be found in Appendix \ref{Lemma_3_exponencial_proof}. Lemma \ref{Lemma_3_exponencial} allows us to write the optimal state of problem \eqref{problem_Jinf} as 
 \begin{align}\label{eq_lim_feedback_backward}
    \begin{cases} 
     \hat{y}_{t} +\mathcal{M}  \hat{y} = 0 \quad & (x,t)\in\Omega\times(0,\infty),\\
     \hat{y}(x,t)=0 &(x,t)\in\partial\Omega\times(0,\infty),\\
     \hat{y}(x,0)=y_0(x) &x\in \Omega,
    \end{cases}
\end{align}
where $\mathcal{M}:=(\mathcal{A} +\chi_\omega \hat{E})$.

The following proposition is of utmost importance, as it allows us to ensure the uniform exponential stabilization of the Riccati operators.

\begin{proposition}\label{corollary_1_exponencial}
Assume the coefficients $a,\,b$ and $p$ are bounded in $\mathscr{C}$. Additionally, assume that \eqref{elipticity_of_a} and \eqref{condition_kernel} are fulfilled. There exist two positive constants $\mu$ and $C_0$ independent of $T$ and the coefficients $(a,b,p)$ such that
\begin{align*}
    \|\mathcal{E}(t)-\hat{E}\|_{\mathcal{L}(L^2(\Omega))}\leq C_0e^{-\mu t},
\end{align*}
for every $t\geq0$ and coefficients $(a,\,b,\,p)$ in this class.
\end{proposition}
The proof of the Proposition \ref{corollary_1_exponencial} relies on the following lemma, whose proof can be found in Appendix \ref{corollary_1_exponencial_lemma_proof}. 
\begin{lemma}\label{corollary_1_exponencial_lemma}
Assume the coefficients $a,\,b$ and $p$ are bounded in $\mathscr{C}$ and that \eqref{elipticity_of_a} and \eqref{condition_kernel} are fulfilled. Then, there exist three positive constants $C_1,\, C_2$ and $ \mu$, independent of $T$ and the coefficients under consideration, such that 
\begin{align*}
   \frac{1}{2}\|\hat{E}y_0(\cdot)\|_{L^2(\Omega)}^2\leq(\hat{E}y_0(\cdot),y_0(\cdot))_{L^2(\Omega)}=\min_{ f\in L^{2}(0,\infty;\Omega)}J^{\infty,0}( f)\leq C_1\|y_0(\cdot)\|^2_{L^2(\Omega)},
\end{align*}
and
\begin{align*}
\| \hat{y}(\cdot,t)\|_{L^2(\Omega)} \leq  C_2e^{-\mu t} \|y_0(\cdot)\|_{L^2(\Omega)},
\end{align*}
where $ \hat{y}$ is the optimal state  of \eqref{eq_lim_feedback_backward}.
\end{lemma}
Note that the last inequality guarantees that the semigroups generated by $\mathcal{M}$ and its adjoint $\mathcal{M}^*$ are uniformly exponentially stable.

\begin{proof}[\bf Proof of the Proposition \ref{corollary_1_exponencial}]
Let us start by considering the optimal pairs $( y,\psi)$ and $( \hat{y},\hat{\psi})$ of problems \eqref{problem_J0}, and \eqref{problem_Jinf}, respectively. Subtracting the equations of the optimal states and integrating on $(0, T)$, we have  
\begin{align}\label{eq_prop_proof_41_1}
\begin{split}
    \int_0^T\|  y(\cdot,t)- \hat{y}(\cdot,t)\|^2_{L^2(\Omega)}  +  \|\chi_\omega ( \psi(\cdot,t)-\hat{\psi}(\cdot,t))\|^2_{L^2(\Omega)}dt&= -\int_\Omega  ( y(x,T)- \hat{y}(\cdot,T)) (\hat{\psi}(\cdot,T)) dx\\
     &\leq \| y(x,T)- \hat{y}(\cdot,T)\|_{L^2(\Omega)}\|\hat{\psi}(\cdot,T)\|_{L^2(\Omega)}.
     \end{split}
\end{align}
Now, applying Lemma \ref{lemma_1} to the difference $ y(x,t)- \hat{y}(x,t)$, there exists $C>0$ independent of $T$ and the coefficients such that
\begin{align}\label{eq_prop_proof_41_2}
    \| y(\cdot,T)- \hat{y}(\cdot,T)\|^2_{L^2(\Omega)}\leq C\biggr(\int_0^T\|  y(\cdot,t)- \hat{y}(\cdot,t)\|^2_{L^2(\Omega)}  +  \|\chi_\omega ( \psi(\cdot,t)-\hat{\psi}(\cdot,t))\|^2_{L^2(\Omega)}dt\biggr).
\end{align}
Hence, combining \eqref{eq_prop_proof_41_1}-\eqref{eq_prop_proof_41_2} we get
\begin{align*}
     \int_0^T\|  y(\cdot,t)- \hat{y}(\cdot,t)\|^2_{L^2(\Omega)}  +  \|\chi_\omega ( \psi(\cdot,t)-\hat{\psi}(\cdot,t))\|^2_{L^2(\Omega)}dt\leq C\|\hat{\psi}(\cdot,T)\|_{L^2(\Omega)}^2.
\end{align*}
Since $\hat{\psi}(t)=\hat{E}  \hat{y}(t)$, applying Lemma \ref{lemma_1} for the difference $\psi(x,t)-\hat{\psi}(x,t)$, and Lemma \ref{corollary_1_exponencial_lemma}, we deduce that 
\begin{align}\label{eq_prop_proof_41_3}
\begin{split}
     \|\psi(\cdot,0)-\hat{\psi}(\cdot,0)\|^2_{L^2(\Omega)}&\leq C\biggr(\int_0^T\|  y(\cdot,t)- \hat{y}(\cdot,t)\|^2_{L^2(\Omega)}  +  \|\chi_\omega ( \psi(\cdot,t)-\hat{\psi}(\cdot,t))\|^2_{L^2(\Omega)}dt\biggr)\\
     & \leq C\|\hat{\psi}(\cdot,T)\|_{L^2(\Omega)}^2\leq 2CC_1C_2^2 e^{-2\mu T}\|y_0\|^2_{L^2(\Omega)}.
     \end{split}
\end{align}
Finally, we conclude the proof by using the definitions of $\mathcal{E}$ and $\hat{E}$ on the left-hand side of \eqref{eq_prop_proof_41_3}.
\end{proof}

Returning to the original problem \eqref{evolutive_control_problem} with the original target $y_d$, not necessarily the null one, the following result enables the optimal control to be characterized in a feedback form.
\begin{corollary}\label{corollary_2_exponencial}
Assume the coefficients $a,\,b$ and $p$ are bounded in $\mathscr{C}$ and that \eqref{elipticity_of_a} and \eqref{condition_kernel} are fulfilled. The optimal control $ f$ of \eqref{evolutive_control_problem} is given by the feedback control law
\begin{align}\label{affine_feedback_law}
     f(x,t)= \overline{f}(x)-\chi_\omega\left[\mathcal{E}(T-t)\left( y(x,t)-\overline{y}(x)\right)+h(x,t)\right],
\end{align}
where $h$ satisfies the following equation
\begin{align}\label{equation_h}
    \begin{cases}
        -h_{t}+\left(\mathcal{A}^* +\chi_\omega\mathcal{E}(T-t)\right) h=0  \quad & (x,t)\in\Omega\times(0,T) \\
        h(x,t)=0 &(x,t)\in\partial\Omega\times(0,T),\\
        h(x,T)=-\overline{\psi}(x)&x\in \Omega.
    \end{cases}
\end{align}
Here $\overline{\psi}(x)$ is the stationary adjoint state.
\end{corollary}

\begin{proof}[\bf Proof of the Corollary  \ref{corollary_2_exponencial}]
Recall that when considering $\eta \in L^2(\Omega)$, the function $q$ defined by $\mathcal{E}(T-t)\eta=q(t)$ satisfies the system
\begin{align*}
    \begin{cases}
    u_{t} +\mathcal{A} u= -\chi_\omega q \quad & (x,t)\in\Omega\times(t,T),\\
    -q_{t} +\mathcal{A}^* q= u  \quad & (x,t)\in\Omega\times(t,T),\\
    u(x,t)=q(x,t)=0 &(x,t)\in\partial\Omega\times(t,T),\\
    u(x,t)=\eta(x),\, q(x,T)= 0 &x\in \Omega.
    \end{cases}
\end{align*}
Then, the new states $m= y-\overline{y}$
and $n=\psi-\overline{\psi}$ solve
\begin{align}\label{equationn_m}
    \begin{cases}
    m_{t} +\mathcal{A} m= -\chi_\omega n \quad & (x,t)\in\Omega\times(0,T),\\
    -n_{t} +\mathcal{A}^* n= m  \quad & (x,t)\in\Omega\times(0,T),\\
    m(x,t)=n(x,t)=0 &(x,t)\in\partial\Omega\times(0,T),\\
    m(x,0)=y_0(x)-\overline{y}(x),\, n(x,T)= -\overline{\psi}(x) &x\in \Omega.
    \end{cases}
\end{align}
Now, using the equation \eqref{equation_h} and the definition of $\mathcal{E}$ for every $\eta\in L^2(\Omega)$, the following relation holds:
\begin{align}\label{variational_relation}
    (n(\cdot,t),\eta(\cdot))_{L^2(\Omega)}=(m(\cdot,t),q(\cdot,t))_{L^2(\Omega)}+(h(\cdot,t),\eta(\cdot))_{L^2(\Omega)}.
\end{align}
Equation \eqref{variational_relation} can be obtained by multiplying the first equation of \eqref{equationn_m} by $u$, integrating over $(t,T)\times \Omega$, and integrating by parts in space and time:
\begin{multline}\label{eq_1}
    (n(\cdot,t),\eta(\cdot))_{L^2(\Omega)}\\=-(\overline{\psi}(\cdot),u(\cdot,T))_{L^2(\Omega)}+\int_t^T(n(\cdot,s),\chi_\omega q(\cdot,s))_{L^2(\Omega)}ds
    +\int_t^T(m(\cdot,s),u(\cdot,s))_{L^2(\Omega)}ds.
\end{multline}
Multiplying by $u$ the first equation in \eqref{equation_h}, integrating over $(t,T)\times \Omega$, and integrating by part in space and time, we have  
\begin{multline*}
    (h(\cdot,t),\eta(\cdot))_{L^2(\Omega)}\\=-(\overline{\psi}(\cdot),u(\cdot,T))_{L^2(\Omega)}+\int_t^T(h(\cdot,s),\chi_\omega q(\cdot,s))_{L^2(\Omega)}ds-\int_t^T(\chi_{\omega} \mathcal{E}(T-s)h(\cdot,s),u(\cdot,s))_{L^2(\Omega)}ds.
\end{multline*}
Taking into account that $\mathcal{E}(T-s)u(s)=q(s)$ for every $s\in(t,T)$, the above equality reduces to
\begin{align}\label{eq_2}
    (h(\cdot,t),\eta(\cdot))_{L^2(\Omega)}=-(\overline{\psi}(\cdot),u(\cdot,T))_{L^2(\Omega)}.
\end{align}
Multiplying by $q$ the first equation of \eqref{equationn_m}, integrating over $(t,T)\times \Omega$, and integrating by parts we get
\begin{align}\label{eq_3}
    (m(\cdot,t),q(\cdot,t))_{L^2(\Omega)}=\int_t^T(m(\cdot,s),u(\cdot))_{L^2(\Omega)}ds+\int_t^T(n(\cdot,s),\chi_{\omega}q(\cdot,s))_{L^2(\Omega)}ds.
\end{align}
Thus, combining equations \eqref{eq_1}-\eqref{eq_3}, we conclude \eqref{variational_relation}. Returning to the original variables, we have 
\begin{align}\label{eq_4}
     (\psi(\cdot,t)-\overline{\psi}(\cdot),\eta(\cdot))_{L^2(\Omega)}=( y(\cdot,t)-\overline{y}(\cdot),\mathcal{E}(T-t)\eta)_{L^2(\Omega)}+(h(\cdot,t),\eta(\cdot))_{L^2(\Omega)}.
\end{align}
Since $\eta$ is arbitrary, from \eqref{eq_4} we obtain
 \begin{align*}
     \int_\Omega  \biggr(\psi(x,t)-\overline{\psi}(x)-\mathcal{E}(T-t)( y(x,t)-\overline{y}(x))-h(x,t)\biggr)\eta(x) dx=0, \quad \forall \eta \in L^2(\Omega).
 \end{align*}
This concludes the proof.
\end{proof}

We are now ready to prove the main theorem.

\begin{proof}[ \bf Proof of Theorem \ref{TH_TURNPIKE_EXP}]
Since $\mathcal{M}^*=(a^* +\chi_\omega \hat{E})$, we can write  equation \eqref{equation_h} as
\begin{align}\label{equation_h2}
    \begin{cases}
        -h_{t}+\mathcal{M}^*h+(\mathcal{E}(T-t)-\hat{E}) \chi_\omega h=0  \quad & (x,t)\in\Omega\times(0,T), \\
        h(x,t)=0 &(x,t)\in\partial\Omega\times(0,T),\\
        h(x,T)=-\overline{\psi}&x\in \Omega.
    \end{cases}
\end{align}
Denote by $S(t)$ the $C^0-$semigroup generated by $\mathcal{M}$, and $S^*(t)\in \mathcal{L}(L^2(\Omega))$ the adjoint of $S(t)$. Since $L^2(\Omega)$ is a Hilbert space, then $\mathcal{M}^*$ is the generator of the $C^0-$semigroup $S^*(t)$. Thus, the solution of \eqref{equation_h2} is given by
\begin{align*}
    h(x,t)=-S^*(T-t)\overline{\psi}(x) + \int_t^T S^*(s-t)\left((\mathcal{E}(T-s)-\hat{E})\chi_\omega h(s)\right)ds.
\end{align*}
Using Proposition \ref{corollary_1_exponencial}, Corollary \ref{corollary_2_exponencial}, and Lemma \ref{lemma_1} we obtain
\begin{align*}
    \|h(t)\|_{L^2(\Omega)}&\leq C_2\| y_d\|_{L^2(\Omega)}e^{-\mu(T-t)} + C_0C_2\int_t^T e^{-\mu(s-t)} e^{-\mu(T-s)}\|h(s)\|_{L^2(\Omega)}ds\\
   &\leq C_2\max\{1,C_0\} \left(\|y_d\|_{L^2(\Omega)}e^{-\mu(T-t)} + e^{-\mu T}\int_t^T  e^{\mu(t+s)}\|h(s)\|_{L^2(\Omega)}ds\right).
\end{align*}
Applying Gronwall’s lemma, we have 
\begin{align}\label{stability_h}
\begin{split}
       \|h(t)\|_{L^2(\Omega)}&\leq  C_2\max\{1,C_0\}\|y_d\|_{L^2(\Omega)}e^{-\mu(T-t)} \exp\left(C_2\max\{1,C_0\}\int_t^T e^{-\mu T} e^{\mu(t+s)} ds\right)\\
     &\leq C_3\|y_d\|_{L^2(\Omega)}e^{-\mu(T-t)},
\end{split}
\end{align}
for every $t\in [0,T]$, where 
\begin{align}\label{constant_c3}
C_3=C_2\max\{1,C_0\}.    
\end{align}
On the other hand, we can observe that $z= y-\overline{y}$ satisfies
\begin{align}\label{z_system_proof_main_theorem}
     \begin{cases}
    z_t + \mathcal{M} z=\chi_\omega \left((\hat{E}-\mathcal{E}(T-t))z-h\right)\quad & (x,t)\in\Omega\times(0,T),\\
    z(x,t)=0 &(x,t)\in\partial\Omega\times(0,T),\\
    z(x,0)=y_0(x)-\overline{y}(x)&x\in \Omega.
    \end{cases}
\end{align}
Then, the solution of \eqref{z_system_proof_main_theorem} is given by
\begin{align*}
    z(x,t)=S(t)(y_0(x)-\overline{y}) +\int_0^t S(t-s)\chi_\omega \left((\hat{E}-\mathcal{E}(T-t))z(x,s)-h(x,s)\right) ds.
\end{align*}
 Now using estimate \eqref{stability_h} and Lemma \ref{lemma_1}, we obtain
\begin{align*}
    \|z(t)\|_{L^2(\Omega)}&\leq\begin{multlined}[t][9cm]C_2(\|y_0\|_{L^2(\Omega)}+\hat{C}\|y_d\|_{L^2(\Omega)})e^{-\mu t} +C_0C_2\int_0^t e^{-\mu(t-s)}e^{-\mu(T-s)}\|z(s)\|_{L^2(\Omega)} ds\\ + C_2C_3\int_0^t e^{-\mu(t-s)}e^{-\mu(T-s)}ds\end{multlined} \\
    &\leq\begin{multlined}[t][12cm]C_2\left((\|y_0\|_{L^2(\Omega)}+\hat{C}\|y_d\|_{L^2(\Omega)})e^{-\mu t} + C_3\|y_d\|_{L^2(\Omega)}e^{-\mu (T-t)}\right)\\
    +C_0C_2e^{-\mu(T-t)}\int_0^t e^{-\mu(t-s)}e^{-\mu(T-s)}\|z(s)\|_{L^2(\Omega)} ds
    \end{multlined}\\
    &\leq C_4\left(\left((\|y_0\|_{L^2(\Omega)}+\|y_d\|_{L^2(\Omega)})e^{-\mu t} + \|y_d\|_{L^2(\Omega)}e^{-\mu (T-t)}\right) +\int_0^t e^{- \mu (t-s)}\|z(s)\|_{L^2(\Omega)} ds\right),
\end{align*}
where 
\begin{align}\label{constant_c4}
    C_4=C_2\max\{1,\hat{C},C_0,C_3\}.
\end{align}
Applying Gronwall's lemma again, we deduce
\begin{align*}
 \| y(\cdot,t)-\overline{y}(\cdot)\|_{L^2(\Omega)} &\leq 
 \begin{multlined}[t][11cm]
C_4\biggr((\|y_0\|_{L^2(\Omega)}+\|y_d\|_{L^2(\Omega)})e^{-\mu t} + \|y_d\|_{L^2(\Omega)}e^{-\mu (T-t)}\biggr)
 \exp\left(C_4\int_0^t e^{-\mu(t-s)}ds\right)
\end{multlined}\\
 &\leq C_4\left((\|y_0\|_{L^2(\Omega)}+\|y_d\|_{L^2(\Omega)})e^{-\mu t} + \|y_d\|_{L^2(\Omega)}e^{-\mu (T-t)}\right).
\end{align*}
Finally, using the fact that the optimal control $ f$ of \eqref{evolutive_control_problem} is given by the affine feedback law \eqref{affine_feedback_law} and  applying Proposition \ref{corollary_1_exponencial}, we conclude 
\begin{align*}
    \| f(\cdot,t)- \overline{f}(\cdot)\|_{L^2(\Omega)}&\leq (\|\mathcal{E}(T-t)\|_{\mathcal{L}(L^2(\Omega))}\| y(\cdot,t)-\overline{y}(\cdot)\|_{L^2(\Omega)}+\|h(\cdot,t)\|_{L^2(\Omega)})\\
    &\leq C_5\left((\|y_0\|_{L^2(\Omega)}+\|y_d\|_{L^2(\Omega)})e^{-\mu t} + \|y_d\|_{L^2(\Omega)}e^{-\mu (T-t)}\right),
\end{align*}
for every $t\in (0,T)$, where 
\begin{align*}
C_5=\max\{(C_0+(2C_1)^{1/2}\|y_0\|_{L^2(\Omega)})C_4,C_3\}.
\end{align*}
\end{proof}

\begin{remark}\label{explicit_constant_turnpike}
\begin{enumerate}
    \item The constants $C_0,\,C_1,\,C_2$ and $\hat{C}$ in the previous proof are given by Proposition \ref{corollary_1_exponencial}, Lemma \ref{corollary_1_exponencial_lemma} and Lemma \ref{lemma_1}, respectively. The  constant $C$ of  \eqref{ineq_exponential_turnpike} is given by $C=2\max\{C_4,C_5\},$ with $C_3$ and $C_4$ given by \eqref{constant_c3} and \eqref{constant_c4}, respectively. These constants are independent of $T$ and the coefficients $(a,b,p)$ in the class considered.

    \item In the previous lemmas, propositions, and theorems, we assumed that the coefficients $(a, b, p)$ are bounded in $\mathscr{C}$. However, these uniform bounds on the coefficients are only needed to ensure the uniform null control property. This observation directly implies Theorem \ref{one_dimensional_UT}.
\end{enumerate}

\end{remark}

\section{Numerical simulations}\label{numerical}
In this section, we present some numerical experiments to confirm the uniform turnpike property in the context of homogenization. Let us fix the time horizon $T=50$ and the $1-d$ domain $\Omega=(0,1)$. Consider the optimal control problem
\begin{align*}
    \min_{ f^{\varepsilon}\in L^{2}(0,T;(0,1))}\biggr\{J_\varepsilon^T( f)=\frac{1}{2}\int_0^T \| f^{\varepsilon}\|_{L^2(0,1)}^2 dt +\| y^{\varepsilon}(\cdot,t)-1\|^2_{L^2(0,1)}dt\biggr\},
\end{align*}
where $y^{\varepsilon}$ is the solution of \eqref{rapidly_heat} with $a(\cdot)\in L^\infty(\R)$ a periodic function given by \eqref{homogenizated_a} satisfying \eqref{bound_a_homo}, and $b\equiv p\equiv0$.

On the other hand, consider the corresponding stationary optimization problem
\begin{align*}
    \min_{\overline{f}^{\varepsilon}\in L^{2}(0,1)}\biggr\{ J^{s}(\overline{f}^{\varepsilon})=\frac{1}{2}\biggr(\|\overline{f}^{\varepsilon}(\cdot)\|^{2}_{L^{2}(0,1)}+\|\overline{y}^{\varepsilon}(\cdot)-1\|^{2}_{L^{2}(0,1)}\biggr)\biggr\},
\end{align*}
where $\overline{y}^{\varepsilon}$ denotes the solution of the rsteady state system \eqref{stationry_system_2}. In this setting, the uniform turnpike property, i.e. Corollary \ref{corollary_uniform_epsilon}, holds.

Consider the particular case
\begin{align*}
    a\left(\frac{x}{\varepsilon}\right)= \sin^2\left(\frac{x\pi}{\varepsilon}\right)+0.5,\quad y_0(x)=x(x-1),
\end{align*}
modelling a heterogeneous material, where the conductivity coefficient oscillates periodically between two different constant ones ($=0.5$ and $1.5$). As the frequency of oscillation increases, materials mix leading to an homogenized homogeneous one in the limit.

In our numerical simulations the semi-discrete problem is dealt with the FEniCS library in Python \cite{MR3618064}, a high-level interface based on finite elements. The backward Euler discretization in time with 168 elements, and the spatial discretization with 421 elements was implemented, leveraging the standard Lagrange family of elements. We then solved the optimal control problem with Dolfin-adjoint \cite{Mitusch2019}. 
This procedure was performed for the various values of the parameter $\varepsilon\in (0.005,1)$. This is illustrated in Figure \eqref{fig1} and \eqref{fig2}.

\begin{figure}
\centering
\subfloat[Optimal states with $\varepsilon=1$ and $\varepsilon=0.5$, respectively. ]{
\includegraphics[width=0.45\textwidth]{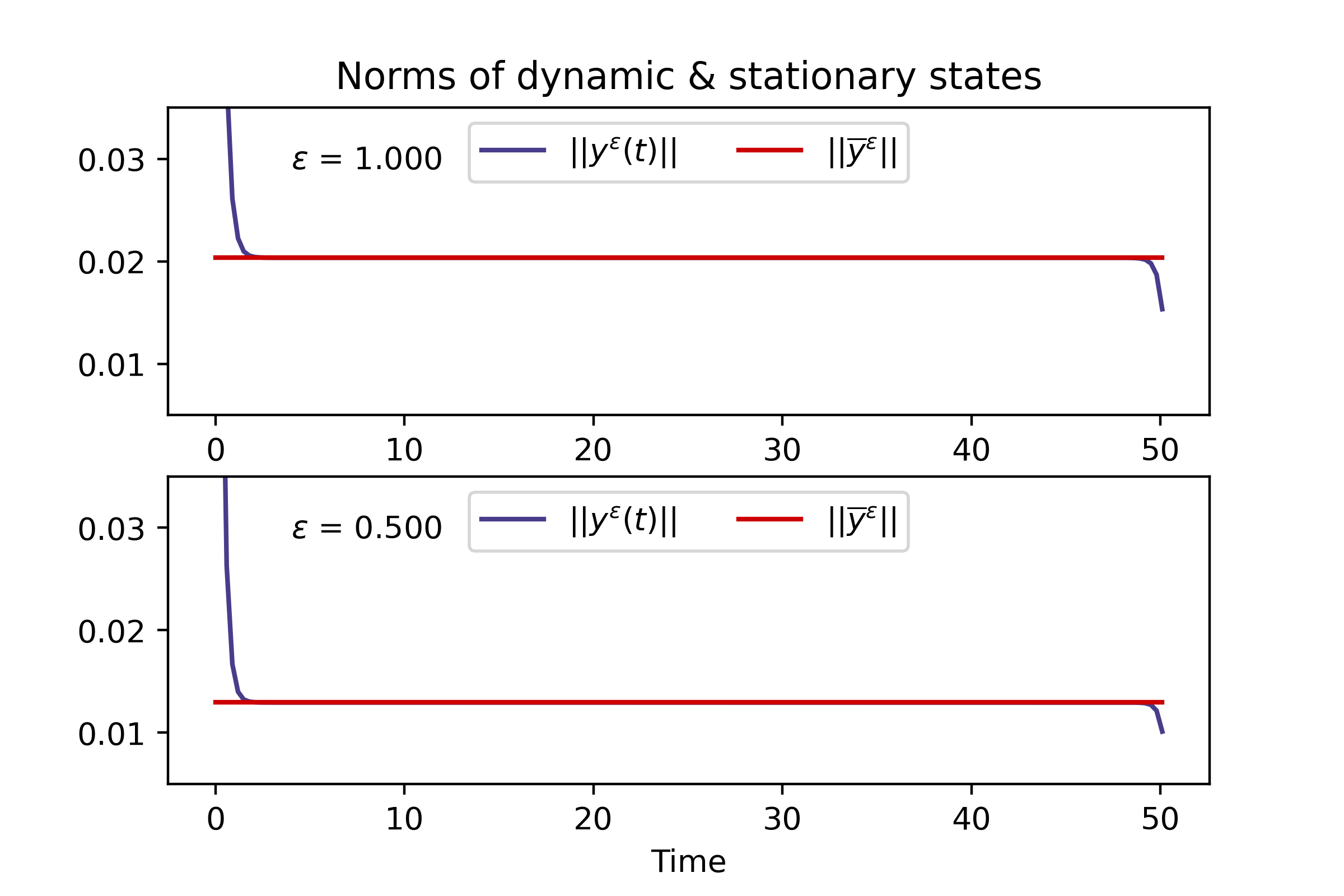}}
\subfloat[Optimal states with $\varepsilon=0.01$ and $\varepsilon=0.005$, respectively.]{
\includegraphics[width=0.45\textwidth]{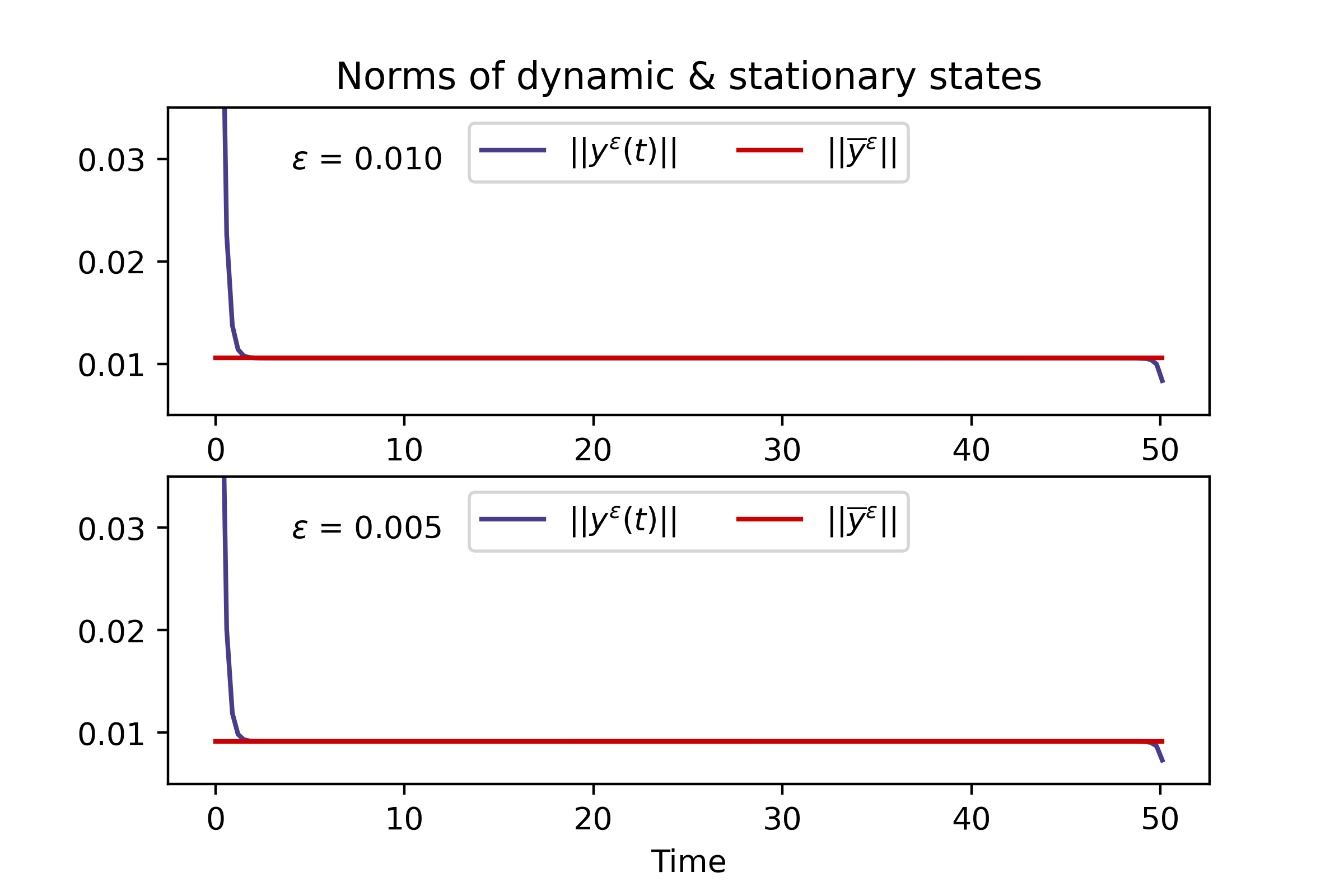}}\caption{Norm of evolutive and stationary states for different values of $\varepsilon$}\label{fig1}
\end{figure}

In Figure \ref{fig1}, the turnpike property is validated for a range of $\varepsilon$ values. Figure \ref{fig2} further illustrates that, for multiple $\varepsilon>0$ values, the uniform exponential turnpike \eqref{ineq_exponential_turnpike} is maintained with fixed constants $C,\,\mu>0$. These constants were obtained experimentally in our numerical tests.
\begin{figure}
    \centering
    \includegraphics[width=0.85\textwidth]{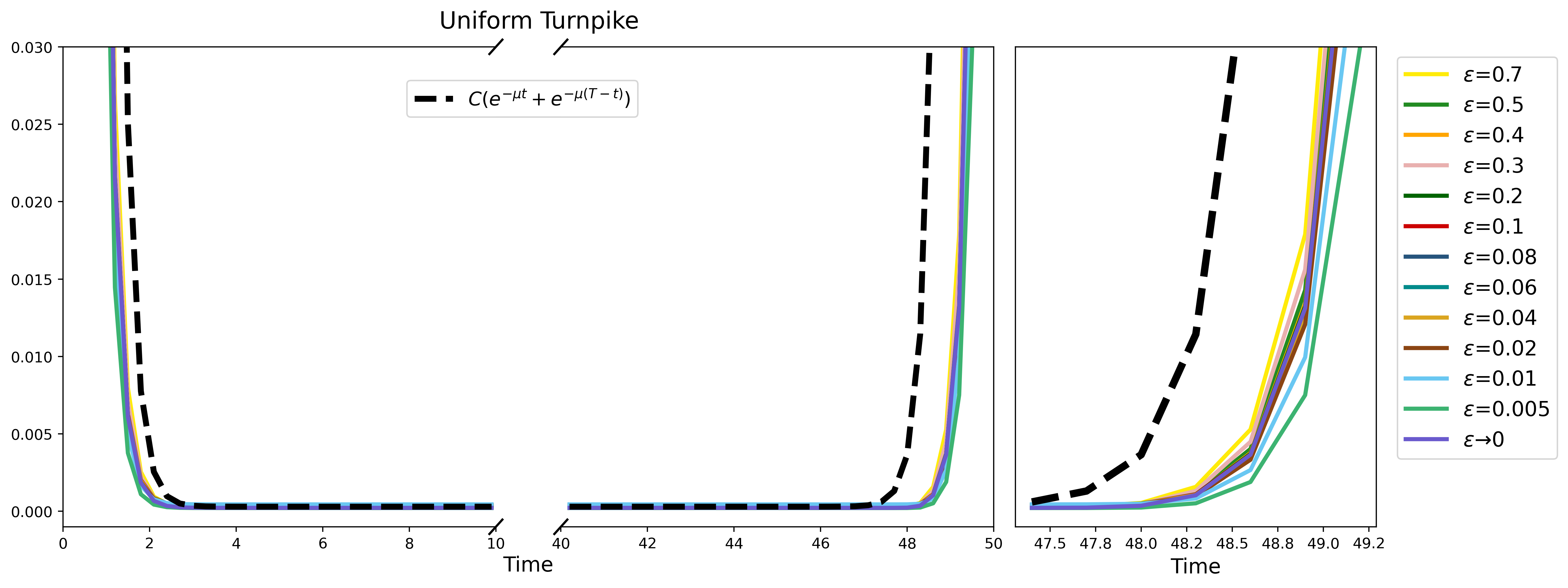}
    \caption{Left: Illustration of the uniform exponential turnpike property for different values of $\varepsilon$ with fixed constants $C=10$ and $\mu=4$. The dashed black line represents the turnpike uniform bound. The colored lines beneath this black line correspond to the quantity $\|y^\varepsilon(t)-\overline{y}^\varepsilon\|_{L^2(\Omega)}+\|f^\varepsilon(t)-\overline{f}^\varepsilon\|_{L^2(\Omega)}$, with different colors representing various values of $\varepsilon$. Right: A zoomed-in view of the uniform turnpike property at the end of the interval $(0,50)$.}\label{fig2}
\end{figure}

Let us analyze the homogenization singular limit problem, when $\varepsilon\to 0$. We know that the homogenized coefficient of  \eqref{lim_system} and  \eqref{lim_system_stationary}, in this $1-d$ example, is given by
\begin{align*}
    a_h= \left(\frac{1}{|\Omega|}\int_{\Omega} \frac{1}{a(x)}dx\right)^{-1}=\left(\int_0^1\frac{1}{sin^2\left(\frac{x\pi}{\varepsilon}\right)+0.5}dx\right)^{-1} .
\end{align*}
In this case, $a_h \approx 0.86603$. The limit optimal control problems \eqref{evolutive_control_problem} and \eqref{stationary_var_problem} are solved, subject to the homogenized equations \eqref{lim_system} and \eqref{lim_system_stationary} with the coefficient $a_h$. The comparison of the limit model with different values of $\varepsilon>0$ is illustrated in Figure \ref{fig2}. 

From Figure  \ref{fig2}, we can also observe that the turnpike property still holds for the homogenized system with the same constants $C$ and $\mu$.

Finally, observe that  the uniform turnpike property and Lemma \ref{lemma_1}, guarantee that
\begin{align*}
    \| y(\cdot,t)\|_{L^2(\Omega)}\leq C\left(\|y_0\|_{L^2(\Omega)}+\|y_d\|_{L^2(\Omega)}\right)(e^{-\mu t}+e^{-\mu(T-t)} + \|y_d\|_{L^2(\Omega)}).
\end{align*}
This shows that all the trajectories are contained in a tubular neighborhood, defined by the turnpike constants and the norm of the target. This can be noticed in Figure \ref{fig4}.

\begin{figure}
\centering
\includegraphics[width=0.65\textwidth]{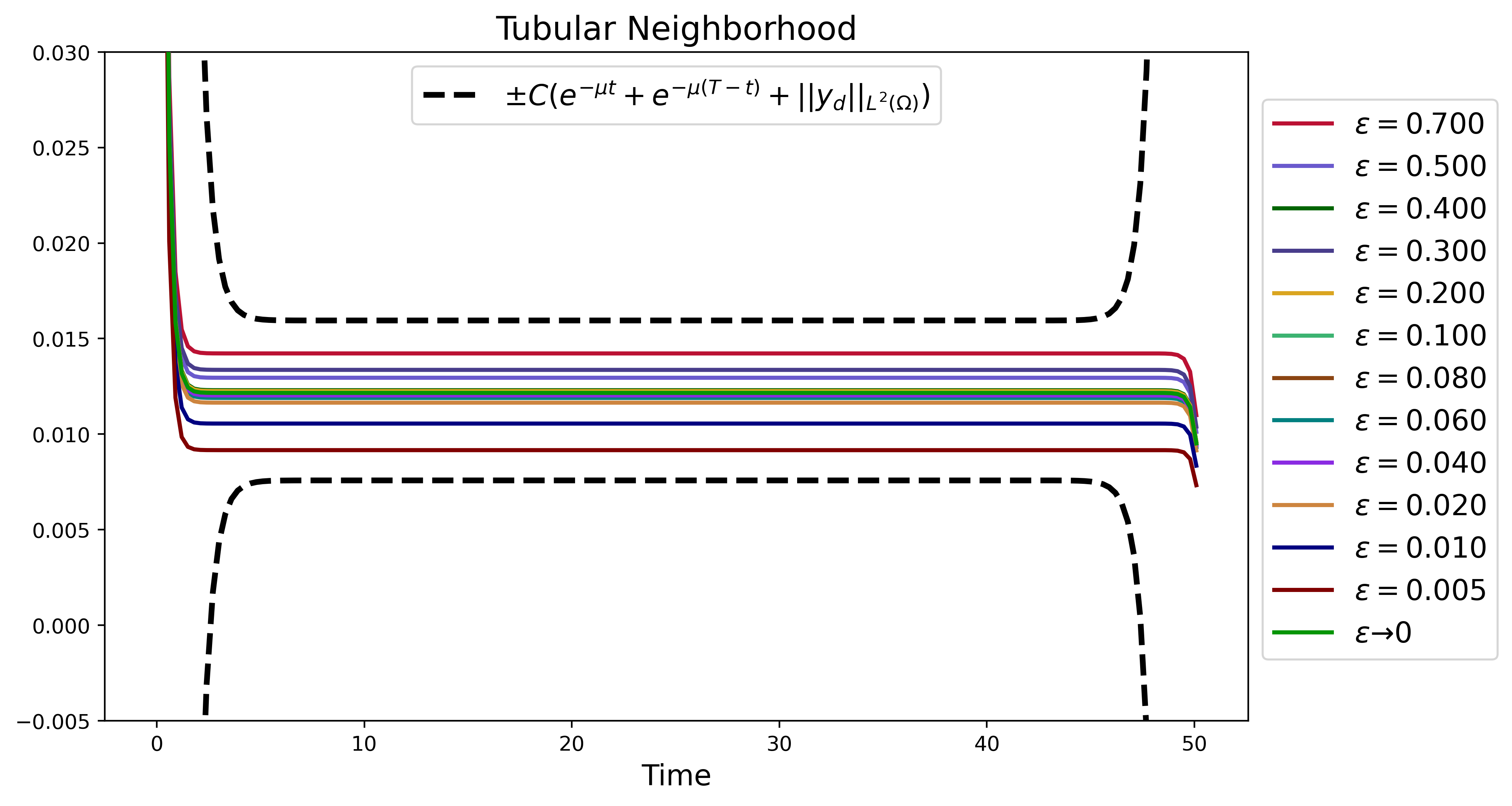}
\caption{Tubular neighborhood, defined by the uniform turnpike constants, containing all the optimal states. Lines of different colors represent different norms $\|y^\varepsilon(t)\|_{L^2(\Omega)}$ of the optimal trajectories, corresponding to various values of $\varepsilon$.}\label{fig4}
\end{figure}

\section{Further comments and open problems}\label{further_comentaries}

Summarizing, in this paper, we show that the uniform turnpike property is a consequence of the uniform null controllability property. Although the uniform turnpike property has been proved for a particular parabolic equation,  this general principle remains valid for other models, such as finite-dimensional ones or wave-type equations, as long as the uniform null control property holds.

In the following, we indicate some interesting open questions related to our analysis.

\begin{enumerate}

\item The analysis of the singular homogenization limit of the turnpike (Corollary \ref{turnpike_limit_states}) was performed in $1-d$. This limitation is due to the lack of uniform controllability results in higher space dimensions without $W^{1,\infty}-$bounds on the coefficients of the principal part. The problem of null-controllability and homogenization of the multi-dimensional heat equation is widely open. So is the case for the uniform turnpike property.

\item It would be interesting to investigate similar questions for the singular limit of the equation
\begin{align}\label{lopez_waves}
     \begin{cases}
   \varepsilon y_{tt}-\Delta y+ y_t = \chi_\omega f &\text{ in }\Omega\times(0,T),\\
     y=0 &\text{ on }\partial\Omega\times(0,T),\\
     y(x,0)=y_0(x),\quad  y_t(x,0)=y_1(x)&\text{ in }\Omega.
    \end{cases}
\end{align}
Observe that the behavior of this equation varies from hyperbolic to parabolic as $\varepsilon \to 0$. The uniform controllability of this model in the singular limit regime was proved by Lopez, Zhang, and Zuazua in \cite{MR1782102}. This result can be used directly to prove the uniform turnpike property, following the same methodology of this article.

Similar questions can be more easily handled when the perturbation enters in lower order terms, such as in
\begin{align*}
  \begin{cases}
    y_{tt}-\Delta y + \varepsilon y_t= f &\text{ in }\Omega\times(0,T),\\
     y=0 &\text{ on }\partial\Omega\times(0,T),\\
     y(x,0)=y_0(x),\quad  y_t(x,0)=y_1(x)&\text{ in }\Omega.
    \end{cases}
\end{align*}
In this example, following \cite{longtime}, in the presence of a control, the turnpike property holds for each $\varepsilon\in(0,1)$ when the cost functional penalizes the state and control in suitable energy spaces. Also, from \cite{TURNPIKE_ONDA}, we know that the limit wave equation satisfies the turnpike property. The techniques of this paper allow to prove the uniform turnpike property as $\varepsilon \to 0$. 

Note however that, in both examples,  it is necessary to assume that the controls acts on a subset of the domain assuring the controllability to hold, which amounts to assume the Geometric Control Condition.

\item Consider the  linear system with a stiff lower-order term
\begin{align}\label{relaxetion_sistem}
    \begin{cases}
    u_t+v_x=0,\quad (x,t)\in \R\times\R,\\
    \epsilon v_t+au_x=-(v-g(u)),
    \end{cases}
\end{align}
where $\varepsilon$ is a small parameter, and $a$ is a fixed positive constant. In the limit when $\varepsilon\to0$, the system \eqref{relaxetion_sistem} can be approximated by the viscous conservation law
\begin{align}\label{conservation_law}
    \begin{cases}
    u_t-au_{xx}+(g(u))_x=0,\quad(x,t)\in\R\times\R,\\
        v=-au_x+g(u).
    \end{cases}
\end{align}
System \eqref{relaxetion_sistem} constitutes a relaxation of \eqref{conservation_law},
and was introduced by Jin and Xin in \cite{MR1322811} to approximate conservation laws. Note that the system \eqref{relaxetion_sistem} can be rewritten as follows 
\begin{align}\label{conservation_law_wave}
    \varepsilon u_{tt}+u_t-au_{xx}+(g(u))_x=0,\quad(x,t)\in\R\times\R,
\end{align}
The extension of the uniform controllability result in \cite{MR1782102} to the nonlinear model \eqref{relaxetion_sistem} constitutes an interesting open problem together with the uniform turnpike property in the singular limit. The turnpike property for semilinear parabolic problems has been analyzed in \cite{remarks}.

\item The turnpike property has been used in greedy algorithms for parametric parabolic equations in \cite{MR3905433}, exploiting the fact that the controlled parabolic dynamics is close to the elliptic one.

On the other hand, in \cite{MR4382586} the authors proposed a greedy algorithm for the Vlasov-Fokker-Planck equation (VFP) 
\begin{align*}
    \varepsilon y_t+v y_x- \phi_x y_v=\frac{1}{\varepsilon} \mathcal{Q} y, \quad x, v \in \Omega=(0, l) \times \mathbb{R}
\end{align*}
where $\mathcal{Q}$ is Fokker Planck operator, $\phi(t, x, z)$ is a given parameter-dependent potential (the parameters $z$ are random, but with known distribution), $\varepsilon>0$ is the so-called Knudsen number and $y=y(x,t,v,z)$ is the probability density distribution
of particles at position $x$ with velocity $v$. The proposed greedy algorithm breaks the curse of dimensionality and allows obtaining better convergence rates than those obtained using Monte-Carlo methods.

The boundary control of this equation has already been studied in \cite{shijin_1}, where it is also proved that this control is uniform with respect to the Knudsen number.

Motivated by the above, it would be natural to analyze the turnpike property for the VFP equation and whether this property is uniform with respect to the Knudsen number or, even more, uniform with respect to the random parameters, making it possible to analyze the singular limit when $\varepsilon\to 0$.
It would also be interesting to investigate the implications from a numerical point of view as in \cite{MR3905433}.

\item In the context of parabolic models, analyzing the property of approximate controllability is also natural. It consists in driving the system arbitrarily close to any target in the phase space $L^2(\Omega)$. This property is uniform in the homogenization context \cite{MR1303384}, so it would be natural to explore further the possible convergence in the context of the turnpike property.

\end{enumerate}

\appendix
\section{Proof of technical results}\label{Appendix_lemmas}

 Consider the operators $\mathcal{A}$ and $\mathcal{A}^*$ introduced in \eqref{operator_A} and \eqref{operator_A_adjoint}, respectively. If we consider the coefficients $(a,b,p)\in L^\infty(\Omega)\times (L^\infty(\Omega))^n\times L^\infty(\Omega)$ (uniform) bounded, and assume \eqref{elipticity_of_a}, then is straightforward to see that
\begin{align}\label{elipticidad_A}
     (\mathcal{A} y, y)_{L^2(\Omega)}+C_r\| y\|_{L^2(\Omega)}^2\geq \frac{a_0}{2}\| y\|_{H_0^1(\Omega)}^2,
\end{align}
and 
\begin{align}\label{elipticidad_A_star}
 (\mathcal{A}^*\psi,\psi)_{L^2(\Omega)}+    C_r\|\psi\|_{L^2(\Omega)}^2\geq \frac{a_0}{2}\|\psi\|_{H_0^1(\Omega)}^2,
\end{align}
where $C_r$ is a positive constant such that
\begin{align}\label{constant_Cr}
    \left(\frac{\|b\|_{L^{\infty}(\Omega)}^2}{2a_0}+\|p\|_{L^{\infty}(\Omega)}\right)\leq C_r.
\end{align}

On the other hand, both Theorem \ref{teo_controlabilidad_uniforme_0} and \ref{theorema_1d_null_control} are proved by duality; that is by establishing the observability inequality 
\begin{align}\label{observability_inequality}
    \|\phi(\cdot,0)\|^{2}_{L^{2}(\Omega)}\leq C_T\int_{0}^{T}\|\chi_\omega\phi(\cdot,t)\|^{2}dxdt,
\end{align}
where $C_T$ is the controllability cost, and $\phi$ satisfies the equation 
\begin{align}\label{adjoin_control_var}
\begin{cases}
    -\phi_t +\mathcal{A}^*\phi=0\quad & (x,t)\in\Omega\times(0,T),\\
    \phi(x,t)=0 &(x,t)\in\partial\Omega\times(0,T),\\
    \phi(x,T)=\phi_0&x\in \Omega.
    \end{cases}
\end{align}
Since Theorems \ref{teo_controlabilidad_uniforme_0} and \ref{theorema_1d_null_control} ensure uniform controllability, the cost of controllability $C_T$ in \eqref{observability_inequality} is independent of the parameters $(a,b,p)$, but dependent of the time horizon $T$ and the ellipticity constant $a_0$.
The inequality \eqref{observability_inequality} will be useful in the following proof.

\addtocontents{toc}{\protect\setcounter{tocdepth}{1}}
\subsection{Proof of Lemma \ref{lemma_1}}\label{lemma_1_proof}

We divide the proof into three steps, one for each inequality.

{\bf Step 1:} Let us multiply the equation \eqref{rapidly_ocilation_general_heat}, by its solution $ y$ and integrate on $\Omega$
\begin{align*}
    \frac{1}{2}\frac{d}{dt}\left(\| y(\cdot,t)\|^2_{L^2(\Omega)}\right)+(\mathcal{A} y(\cdot,t), y(\cdot,t))_{L^2(\Omega)}=(\chi_{\omega} f(\cdot,t), y(\cdot,t))_{L^2(\Omega)}.
\end{align*}
Then, applying \eqref{elipticidad_A} and the Young's inequality with $\delta>0$, we have
\begin{align*}
    \frac{1}{2}\frac{d}{dt}\left(\| y(\cdot,t)\|^2_{L^2(\Omega)}\right)+&\frac{a_0}{2}\| y(\cdot,t)\|_{H_0^1(\Omega)}^2\\
    &\leq\frac{1}{2\delta}\|\chi_{\omega} f(\cdot,t)\|^2_{L^2(\Omega)}+\frac{\delta}{2}\| y(\cdot,t)\|^2_{L^2(\Omega)}+C_r\| y(\cdot,t)\|_{L^2(\Omega)}^2.
\end{align*}
Now, integrating on the time interval $(0,T)$ and applying Poincar\'e\textquotesingle s inequality, 
\begin{align*}
   \| y(\cdot,T)\|^2_{L^2(\Omega)}+(a_0 C_p-\delta)&\int_0^T\| y(\cdot,t)\|_{L^2(\Omega)}^2\\
   &\leq\frac{1}{\delta}\int_0^T\left(\|\chi_{\omega} f(\cdot,t)\|^2_{L^2(\Omega)}+2C_r\| y(\cdot,t)\|_{L^2(\Omega)}^2\right) dt+  \|y_0(\cdot)\|^2_{L^2(\Omega)}.
\end{align*}
Finally, we conclude the inequality by taking $\delta< a_0 C_p$, where $C_p$ is the Poincar\'e constant.

{\bf Step 2:} Let $I^*=(0,1)\subset(0,T)$ and take $\psi=p+q$ in the interval $I^*$, where $p$ and $q$ satisfy
\begin{align*}
      \begin{cases}
    -p_t +\mathcal{A}^* p=0\quad & \Omega\times I^*,\\
    p(x,t)=0 &\partial\Omega\times I^*,\\
    p(x,1)= \psi(x,1)&\Omega,
    \end{cases}\qquad
      \begin{cases}
    -q_t +\mathcal{A}^* q= y-y_d\quad & \Omega\times I^*,\\
    q(x,t)=0 &\partial\Omega\times I^*,\\
    q(x,1)=0& \Omega.
    \end{cases}
\end{align*}
Observe that $p$ solves the system \eqref{adjoin_control_var}, and therefore, $p$ satisfies the observability inequality
\begin{align*}
    \|p(\cdot,0)\|^{2}_{L^{2}(\Omega)}\leq C\int_{I^*}\|\chi_\omega p(\cdot,t)\|^{2}_{L^2(\Omega)}dt,
\end{align*}
with $C$ independent of $T$  and the coefficients $(a,b,p)$. Let us consider the change of variable $\hat{q}=qe^{C_r t}$ where $C_r$ given by \eqref{constant_Cr}. Then $\hat{q}$ satisfies
\begin{align*}
     \begin{cases}
    -\hat{q}_t +\mathcal{A}^* \hat{q}+C_r\hat{q} =e^{C_r t}(y-y_d) & \Omega\times I^*,\\
    q(x,t)=0 &\partial\Omega\times I^*,\\
    q(x,1)=0& \Omega.
    \end{cases}
\end{align*}
Then, multiplying by $\hat{q}$ the equation satisfied by $\hat{q}$ and integrating on $\Omega\times I^*$ we obtain
\begin{align*}
    \frac{\|q(\cdot,0)\|^2_{L^{2}(\Omega)}}{2} +\int_{I^*}((a+IC_r )\hat{q}(\cdot,t),&\hat{q}(\cdot,t))_{L^2(\Omega)}dt = \int_{I^*} e^{C_r t} ( y(\cdot,t) -y_d(\cdot),\hat{q}(\cdot,t))_{L^2(\Omega)}dt \\
    &\leq \frac{e^{C_r}}{2\delta_1}\int_{I^*}\| y(\cdot,t)-y_d\|^2_{L^2(\Omega)} +\frac{\delta_1}{2}\|\hat{q}(\cdot,t)\|^2_{L^2(\Omega)}dt,
\end{align*}
where we use the Young inequality with $\delta_1>0$. Then, applying \eqref{elipticidad_A_star} and Poincar\'e\textquotesingle s inequality on the left-hand side, we have
\begin{align*}
    \|q(\cdot,0)\|^2_{L^{2}(\Omega)} +(a_0C_p-\delta_1)\int_{I^*}\|q(\cdot,t)\|_{L^2(\Omega)}^2dxdt\leq \frac{e^{C_r}}{\delta_1}\int_{I^*}\| y(\cdot,t)-y_d\|^2_{L^2(\Omega)},
\end{align*}
provided $\delta_1<a_0C_p$. We conclude that there exist $C_\delta>0$ such that
\begin{align*}
      \|q(\cdot,0)\|^2_{L^{2}(\Omega)} +\int_{I^*}\|q(\cdot,t)\|_{L^2(\Omega)}^2dt\leq C_{\delta}\int_{I^*}\| y(\cdot,t)-y_d\|^2_{L^2(\Omega)}.
\end{align*}
Finally returning to the variable $\psi$, we conclude that there exists a constant $C>0$ independent of $T>0$ and the coefficients $(a,b,p)$ such that 
\begin{align*}
    \|\psi(\cdot,0)\|^2_{L^2(\Omega)}
    &\leq 2\left(\|p(\cdot,0)\|^2_{L^2(\Omega)}+\|q(\cdot,0)\|^2_{L^2(\Omega)}\right)\\
    &\leq C\int_{I^*}\|\chi_\omega p(\cdot,t)\|^{2}dt + C_{\delta}\int_{I^*}\| y(\cdot,t)-y_d\|^2_{L^2(\Omega)}dt\\
    &\leq \begin{multlined}[t][12cm]
    \hat{C}_{\delta}\biggr(\int_{I^*}\|\chi_\omega\psi(\cdot,t)\|^{2}_{L^2(\Omega)}dt+ \int_{I^*}\|\chi_\omega q(\cdot,t)\|_{L^2(\Omega)}^{2}dt+\int_{I^*}\| y(\cdot,t)-y_d\|^2_{L^2(\Omega)}dt\biggr)
    \end{multlined}\\
    &\leq \hat{C}_{\delta}\biggr(\int_{I^*}\|\chi_\omega\psi(\cdot,t)\|^{2}_{L^2(\Omega)}dt+\int_{I^*}\| y(\cdot,t)-y_d\|^2_{L^2(\Omega)}dt\biggr)\\
    &
    \leq C\biggr(\int_{0}^T\|\chi_\omega\psi(\cdot,t)\|^{2}_{L^2(\Omega)}+\| y(\cdot,t)-y_d\|^2_{L^2(\Omega)}dt\biggr).
\end{align*}

{\bf Step 3:} Let us multiply by $\overline{\psi}$ the system satisfied by $\overline{y}$ and integrate on $\Omega$ 
\begin{align*}(a\overline{y}(\cdot),\overline{\psi}(\cdot))_{L^2(\Omega)} =( \overline{f}(\cdot),\chi_\omega\overline{\psi}(\cdot)))_{L^2(\Omega)}=-\| \overline{f}(\cdot)\|_{L^2(\Omega)}^2.
\end{align*}
Analogously, multiplying by $\overline{y}$ the system satisfied by $\overline{\psi}$ and integrating on $\Omega$, yields
\begin{align*}
   (a\overline{y}(\cdot),\overline{\psi}(\cdot))_{L^2(\Omega)} =(\overline{y}(\cdot),\overline{y}(\cdot)-y_d(\cdot))_{L^2(\Omega)}= \|\overline{y}(\cdot)\|^2_{L^2(\Omega)} -(\overline{y}(\cdot),y_d(\cdot))_{L^2(\Omega)}.
\end{align*}
Then, combining the above equalities, we deduce
\begin{align*}
    \|\overline{y}(\cdot)\|^2_{L^2(\Omega)}+ \| \overline{f}(\cdot)\|^2_{L^2(\Omega)} =( \overline{y}(\cdot),y_d(\cdot))_{L^2(\Omega)}.
\end{align*}
Now, using Young's inequality on the right-hand side 
\begin{align*}
    \|\overline{y}(\cdot)\|^2_{L^2(\Omega)}+\| \overline{f}(\cdot)\|^2_{L^2(\Omega)} \leq \frac{1}{2} \|\overline{y}(\cdot)\|^2_{L^2(\Omega)}+\frac{1}{2}\|y_d(\cdot)\|^2_{L^2(\Omega)}.
\end{align*}
Hence, we have
\begin{align}\label{estacion_uniforme}
\|\overline{y}(\cdot)\|^2_{L^2(\Omega)}+\| \overline{f}(\cdot)\|^2_{L^2(\Omega)} \leq  \|y_d(\cdot)\|_{L^2(\Omega)}^2.
\end{align}
Moreover, multiplying by $\overline{y}$ the system satisfied  by $\overline{y}$,   and integrating by parts in $\Omega$, we have
\begin{align*}
    (a\nabla \overline{y},\nabla \overline{y})_{L^2(\Omega)}= -(b\nabla \overline{y},\overline{y})_{L^2(\Omega)}-(p\overline{y},\overline{y})_{L^2(\Omega)}+( \overline{f},\overline{y})_{L^2(\Omega)}.
\end{align*}
Then, using \eqref{elipticity_of_a} and Cauchy–Schwarz's inequality 
\begin{align*}
    a_0\|\nabla \overline{y}\|_{L^2(\Omega)}^2\leq \|b\|_{L^{\infty}(\Omega)}\|\nabla \overline{y}\|_{L^2(\Omega)}\|\overline{y}\|_{L^2(\Omega)}+ \|p\|_{L^{\infty}(\Omega)}\|\overline{y}\|_{L^2(\Omega)}^2+\| \overline{f}\|_{L^2(\Omega)}\|\overline{y}\|_{L^2(\Omega)},
\end{align*}
and thanks to Young’s inequality with $\delta< 2a_0/\|b\|_{L^{\infty}(\Omega)}$, we have
\begin{align*}
    \left(a_{0}-\frac{\delta\|b\|_{L^{\infty}(\Omega)}}{2}\right)\|\nabla \overline{y}\|_{L^2(\Omega)}^2\leq \left(\frac{\|b\|_{L^{\infty}(\Omega)}}{2\delta}+\frac{1}{2}+\|p\|_{L^\infty(\Omega)}\right)\|\overline{y}\|_{L^2(\Omega)}+\frac{\| \overline{f}\|^2_{L^2(\Omega)}}{2}
\end{align*}
Thus, combining \eqref{estacion_uniforme} and the fact that the coefficient $b$ and $p$ are bounded in $(L^\infty(\Omega))^n$ and $L^\infty(\Omega)$, respectively, we deduce that there exists a constant $C>0$ independent of $T>0$ and the coefficients $(a,b,p)$ such that 
\begin{align}\label{estacion_uniforme_2}
\|\overline{y}(\cdot)\|^2_{H^1(\Omega)}+\| \overline{f}(\cdot)\|^2_{L^2(\Omega)} \leq C \|y_d(\cdot)\|_{L^2(\Omega)}^2.
\end{align}
On the other hand, consider $t\in (0,\tau)$ with $\tau$ an arbitrary positive constant. Observe that $\overline{\psi}(x)$ satisfies
\begin{align*}
    \begin{cases}
    -\overline{\psi}_t +\mathcal{A}^* \overline{\psi}=\overline{y}-y_d\quad & \Omega\times (0,\tau),\\
    \overline{\psi}(x,t)=0 &\partial\Omega\times (0,\tau),\\
    \overline{\psi}(x,\tau)= \overline{\psi}&\Omega.
    \end{cases}
\end{align*}
Let us take $\overline{\psi}(x)=p(x)+q(x)$, where
\begin{align*}
      \begin{cases}
    -p_t +\mathcal{A}^* p=0\quad & \Omega\times (0,\tau),\\
    p(x,t)=0 &\partial\Omega\times (0,\tau),\\
    p(x,\tau)= \overline{\psi}&\Omega,
    \end{cases}\qquad
      \begin{cases}
    -q_t +\mathcal{A}^* q=\overline{y}-y_d\quad & \Omega\times (0,\tau),\\
    q(x,t)=0 &\partial\Omega\times (0,\tau),\\
    q(x,\tau)=0& \Omega.
    \end{cases}
\end{align*}
Let us note that $p$ solves the system \eqref{adjoin_control_var}, and therefore, $p$ satisfies the observability inequality
\begin{align*}
    \|p(\cdot,0)\|^{2}_{L^{2}(\Omega)}\leq C_{\tau}\int_{0}^{\tau}\|\chi_\omega p(\cdot,t)\|^{2}_{L^2(\Omega)}dt,
\end{align*}
equivalent to
\begin{align}\label{des_obs_estacionaria}
     \|p(\cdot)\|^{2}_{L^{2}(\Omega)}\leq C_{\tau}\tau\|\chi_\omega p(\cdot)\|^{2}_{L^2(\Omega)}.
\end{align}
Now consider the change of variable $ \eta(x,t)=q(x)e^{tC_r}$, with $t\in (0,\tau)$. Observe that $\overline{\psi}$ satisfies
\begin{align*}
     \begin{cases}
    -\eta_t + \mathcal{A}^* \eta +C_r\eta=e^{tC_r}(\overline{y}-y_d)\quad & \Omega\times (0,\tau),\\
    \eta(x,t)=0 &\partial\Omega\times (0,\tau),\\
    \eta(x,0)= \overline{\psi}  &\Omega,\\
    \eta(x,\tau)= 0 &\Omega,
\end{cases}
\end{align*}
Then, multiplying by $\eta$ the equation satisfied by $\eta$ and integrating on $\Omega\times(0,\tau)$, we have
\begin{align*}
    \|\eta(\cdot,0)\|^2_{L^2(\Omega)}+2\int_{0}^{\tau}&((a+IC_r)\eta(\cdot,t),\eta(\cdot,t))_{L^2(\Omega)}dt=2\int_{0}^{\tau}(e^{tC_r}(\overline{y}(\cdot)-y_d(\cdot)),\eta(\cdot,t))_{L^2(\Omega)}dt
\end{align*}
Applying the inequality \eqref{elipticidad_A_star}, Young's inequality, Poincar\'e\textquotesingle s inequality, and returning to the variable $q$, we obtain
\begin{align*}
    \|q(\cdot)\|^2_{L^2(\Omega)} +& C_p\|q(\cdot)\|_{L^2(\Omega)}^2\int_{0}^{\tau}e^{tC_r}dt\leq e^{2\tau C_r}\tau\frac{\|\overline{y}(\cdot)-y_d(\cdot)\|^2_{L^2(\Omega)}}{\delta}dt+\|q(\cdot)\|^2_{L^2(\Omega)}\delta\int_{0}^{\tau} e^{2tC_r}dt
\end{align*}
Thus, factoring the left-hand side and taking $\delta>0$ small enough, we find that there exists a positive constant $C$, independent of the coefficients $(a,b,p)$ such that
\begin{align}\label{ineq_q_apen}
    \|q(\cdot)\|^2_{L^2(\Omega)}\leq C \|\overline{y}(\cdot)-y_d(\cdot)\|^2_{L^2(\Omega)}.
\end{align}
Then, as in Step 2, combining \eqref{ineq_q_apen} and \eqref{des_obs_estacionaria} we obtain 
\begin{align*}
    \|\overline{\psi}(\cdot)\|_{L^2(\Omega)}^2&\leq 2( \|q(\cdot)\|^2_{L^2(\Omega)}+ \|p(\cdot)\|^2_{L^2(\Omega)}) \\
    &\leq C_{\tau}\left(\|\chi_\omega p(\cdot)\|^{2}_{L^2(\Omega)} + \|\overline{y}(\cdot)-y_d(\cdot)\|^2_{L^2(\Omega)}\right)\\
    &\leq  C_{\tau}\left(\|\chi_\omega \overline{\psi}(\cdot)\|^{2}_{L^2(\Omega)}+ \|\chi_\omega q(\cdot)\|^{2}_{L^2(\Omega)} + \|\overline{y}(\cdot)-y_d(\cdot)\|^2_{L^2(\Omega)}\right)\\
    &\leq  C_{\tau}\left(\|\chi_\omega \overline{\psi}(\cdot)\|^{2}_{L^2(\Omega)} + \|\overline{y}(\cdot)-y_d(\cdot)\|^2_{L^2(\Omega)}\right).
\end{align*}
Then, using the characterization of the optimal control $ \overline{f}$,
\begin{align*}
     \|\overline{\psi}(\cdot)\|_{L^2(\Omega)}^2\leq C(\| \overline{f}\|^2_{L^2(\Omega)}+\|\overline{y}(\cdot)-y_d(\cdot)\|^2_{L^2(\Omega)})
\end{align*}
Moreover, by multiplying by $\overline{\psi}$ the system satisfied by $\overline{\psi}$, and performing the same estimations as above for $\overline{y}$, we have that
\begin{align*}
    \left(a_{0}-\frac{\delta\|b\|_{L^{\infty}(\Omega)}}{2}\right)\|\nabla \overline{\psi}\|_{L^2(\Omega)}^2\leq \left(\frac{\|b\|_{L^{\infty}(\Omega)}}{2\delta}+\frac{1}{2}+\|p\|_{L^\infty(\Omega)}\right)\|\overline{\psi}\|_{L^2(\Omega)}+\frac{\|\overline{y}-y_d\|^2_{L^2(\Omega)}}{2}
\end{align*}
Finally, we conclude the proof by using \eqref{estacion_uniforme} and the previous estimations.

\subsection{Proof of Lemma \ref{Lemma_exponencial}}\label{Lemma_exponencial_proof}

To prove Lemma \ref{Lemma_exponencial}, we divide the proof into two steps.

{\bf Step 1:}
Let $( y, f,\psi)$ the optimal variables of problem \eqref{evolutive_control_problem}. Then we have the inequality
\begin{align}\label{Inequality_necesarry_TurnExp}
\begin{split}
    \int_0^T \| y(\cdot,t)-y_d(\cdot)\|_{L^2(\Omega)}^2dt+\int_0^T \| f(\cdot,t)\|^2_{L^2(\Omega)}dt =(  y(\cdot,0), \psi(\cdot,0))_{L^2(\Omega)}-\int_0^T (y_d(\cdot),  y(\cdot,t)-y_d(\cdot) )_{L^2(\Omega)}dt.
\end{split}
\end{align}
In fact, let us multiply the equation \eqref{rapidly_ocilation_general_heat} by $\psi$, the solution of the adjoint system \eqref{adj_var_op} and integrate by part on space and time. Then, we have
\begin{align*}
        -( y(\cdot,0),\psi(\cdot,0) )_{L^2(\Omega)}  + \int_0^T(  y(\cdot,t), -\psi_t(\cdot,t)- a^{*}\psi(\cdot,t) )_{L^2(\Omega)}dt   =\int_0^T( f(\cdot,t),\chi_{\omega}\psi(\cdot,t))_{L^2(\Omega)}dt.
\end{align*}
Since $\psi$ solves \eqref{adj_var_op} and the optimal control is characterized by $ f(x,t)=-\chi_\omega\psi(x,t)$,
\begin{multline*}
        ( y(\cdot,0),\psi(\cdot,0) )_{L^2(\Omega)}-\int_0^T  \| f (\cdot,t)\|^2_{L^2(\Omega)} dt=\int_0^T( y(x,t),  y(x,t)-y_d(x) )_{L^2(\Omega)}dt  \\
        =\int_0^T\| y(\cdot,t)-y_d(\cdot)\|_{L^2(\Omega)}^2+(y_d(\cdot),  y(\cdot,t)-y_d(\cdot) )_{L^2(\Omega)}dt.
\end{multline*}
Then we have
\begin{multline*}
    \int_0^T \| y(\cdot,t)-y_d(\cdot)\|_{L^2(\Omega)}^2dt+\int_0^T \| f(\cdot,t)\|^2_{L^2(\Omega)}dt \\ =(  y(\cdot,0), \psi(\cdot,0))_{L^2(\Omega)}-\int_0^T (y_d(\cdot),  y(\cdot,t)-y_d(\cdot) )_{L^2(\Omega)}dt.
\end{multline*}

{\bf Step 2:} Let us begin by noting that the operator  $\mathcal{E}(T)$ maps the initial condition of the equation \eqref{evolutive_control_problem} to the final state of the adjoint system \eqref{adj_var_op}. Therefore, based on the well-posedness of the optimality system, $\mathcal{E}(T)$ is well-defined.
Furthermore, $\mathcal{E}(T)$ can be expressed as a composition of two maps: the map $y_0\mapsto \{ y,\psi\}$ and the map $\{ y,\psi\}\mapsto \psi(0)$. These individual maps are known to be linear and continuous according to \cite[Lemma 4.2]{lions_book}. As a result, the operator $\mathcal{E}(T)$ is also linear and continuous.
\begin{enumerate}
    \item From the inequality \eqref{Inequality_necesarry_TurnExp}, with $y_d\equiv0$ we deduce directly the variational characterization
    \begin{align*}
        (\mathcal{E}(T)y_0,y_0)_{L^2(\Omega)}=  \min_{ f\in L^{2}(0,T;\Omega)}J^{T,0}( f).
    \end{align*}
    \item Let us prove that the limit 
\begin{align}\label{limit_E_cursiva}
    \lim_{T\to\infty}(\mathcal{E}(T)y_0,y_0)_{L^2(\Omega)}, 
\end{align}
is finite. For this purpose, observe that $\mathcal{E}(T)$ is bounded and increasing in $T$. In fact, take $t_1\leq t_2$ then, using the variational characterization, it is clear that $\mathcal{E}(t_1)\leq\mathcal{E}(t_2)$. On the other hand, since the system \eqref{rapidly_ocilation_general_heat} is uniformly null controllable at any finite time, then for the time horizon $T=1$, there exists a control $g_0\in L^2(0,T;\Omega)$ driving the system to zero at time $1$, i.e., $ y(\cdot,1)=0$ a.e. in $\Omega$. Furthermore, there exists a constant $C_{r}$ independent of the coefficients such that $\|g_0\|_{L^2(0,1;\Omega)}\leq C_{r}\|y_0\|_{L^2(\Omega)}$. Then the two-step control
    \begin{align}\label{two_step_control}
        \hat{g}(x,t)=\begin{cases}
        g_0(x,t), & \text{if }t\in(0,1),\\
        0,& \text{if }t\in (1,T),
        \end{cases}
    \end{align}
    satisfies that $\|\hat{g}\|_{L^2(0,T;\Omega)}\leq C_{r}\|y_0\|_{L^2(\Omega)}$, and the associated state  satisfies $y(x,t)=0$ for every $t\in[1,T)$. Finally, since $( y, f)$ is the optimal pair, we conclude that
\begin{align*}
(\mathcal{E}(T)y_0(\cdot), y_0(\cdot))_{L^2(\Omega)}=\min_{ f\in L^{2}(0,T;\Omega)}J^{T,0}( f)\leq C_{r}\|y_0(\cdot)\|^2_{L^2(\Omega)},
\end{align*}
which ensures that $\mathcal{E}(T)$ is uniformly bounded. Therefore, the limit \eqref{limit_E_cursiva} is finite.
\end{enumerate}

\subsection{Proof of Lemma \ref{Lemma_3_exponencial}}\label{Lemma_3_exponencial_proof}

 We divide the proof into two steps.

{\bf Step 1: }Let us take a sequence of times $T_n\to \infty$ and consider the optimization problem 
\begin{align}\label{problem_Jinf_ap}
    \min_{ f_n\in L^{2}(0,T_n;\Omega)}\biggr\{ J^{T_n,0}( f_n)=\frac{1}{2}\int_0^{T_n} \| f_n(\cdot,t)\|_{L^2(\Omega)}^2  +\| y_n(\cdot,t)\|^2_{L^2(\Omega)}dt\biggr\},
\end{align}
where $ y_n$ is the solution of \eqref{rapidly_ocilation_general_heat} with time horizon equals to $T_n$. Denote by $( y_n, f_n,\psi_n)$ the triple state-control-adjoint optimal variables of the optimization problem \eqref{problem_Jinf_ap}, where $\psi_n$ is the solution of \eqref{adj_var_op} with time horizon $T_n$. Hereafter we extend by zero the functions $( y_n, f_n,\psi_n)$ in the interval $(T_n,\infty)$. 

Let us recall the infinite time optimization problem
\begin{align}\label{infinity_time_problem}
    \min_{ \hat{f}\in L^{2}(0,\infty;\Omega)}\biggr\{ J^{\infty,0}( \hat{f})=\frac{1}{2}\int_0^\infty \| \hat{f}(\cdot,t)\|_{L^2(\Omega)}^2  +\| \hat{y}(\cdot,t)\|^2_{L^2(\Omega)}dt\biggr\},
\end{align}
where $ \hat{y}$ is solution of \eqref{rapidly_ocilation_general_heat} with time horizon $T=\infty$.

In the first part of the proof of the Lemma \ref{Lemma_3_exponencial}, our goal is to prove that
\begin{align*}
    ( y_n,\psi_n, f_n)\to( \hat{y},\hat{\psi}, \hat{f}),\quad\text{ strongly in}\quad \left(L^2(0,\infty;\Omega)\right)^3. 
\end{align*}

First, using the two-step control \eqref{two_step_control}, we observe that the problem \eqref{infinity_time_problem} is well-defined. Consequently, the optimal variables $( \hat{y}, \hat{f},\hat{\psi})$ are unique and satisfy
\begin{align}\label{system_y*psi*}
    \begin{cases}
     \hat{y}_{t} +\mathcal{A}  \hat{y}= -\chi_\omega\hat{\psi}  \quad & (x,t)\in\Omega\times(0,\infty),\\
    -\hat{\psi}_{t} -\mathcal{A}^*\hat{\psi}=  \hat{y}  \quad & (x,t)\in\Omega\times(0,\infty),\\
     \hat{y}(x,t)=\hat{\psi}(x,t)=0 &(x,t)\in\partial \Omega\times(0,\infty),\\
     \hat{y}(x,0)=y_0(x),\, \hat{\psi}(x,t) \overset{t\to\infty}{\to} 0 &x\in \Omega,\\
     \hat{f}=-\chi_\omega\hat{\psi} & (x,t)\in\Omega\times(0,\infty).
    \end{cases}
\end{align}

Applying Lemma \ref{lemma_1} to the inequality \eqref{Inequality_necesarry_TurnExp} when $y_d\equiv0$, we conclude that there exists a positive constant $C$ independent on $T_n$ for all $n$ such that
\begin{align}\label{inequality_uniform_yn_psin}
    \int_0^{\infty}\| y_n(\cdot,t)\|^2_{L^2(\Omega)}  +  \|\chi_\omega\psi_n(\cdot,t)\|^2_{L^2(\Omega)}dt\leq C.
\end{align}
Then, if we multiply by $\psi$ the equation satisfied by $\psi_n$, integrating by parts, and using the inequality \eqref{inequality_uniform_yn_psin}, see deduce that there exists $C>0$ independent of $n$ such that
\begin{align*}
    \int_0^\infty \|\psi_n(\cdot,t)\|_{H_0^1(\Omega)}^2dt\leq C.
\end{align*}
Carrying out the same procedure for $ y_n$, we have
\begin{align*}
  \int_0^\infty \| y_n(\cdot,t)\|_{H_0^1(\Omega)}^2dt\leq C.
\end{align*}
Therefore, there exist two function $\alpha$ and $\beta$ in $L^2(0,\infty;H_0^1(\Omega))$ such that  $ y_n\wconv\alpha$ and $\psi_n\wconv\beta$ weakly in $L^2(0,\infty;H^1_0(\Omega))$. Now take $\eta\in L^2(0,\infty;H^1_0(\Omega))$, multiplying the equation satisfied by $ y_n$, and integrating on $(0,\infty)\times\Omega$, we get
\begin{multline*}
    \int_0^\infty (( y_n)_t , \eta )_{L^2(\Omega)}dt+ \int_0^\infty(\nabla y_n ,a(x)\nabla\eta )_{L^2(\Omega)}dt+\int_0^\infty(\nabla y_n ,b(x)\eta )_{L^2(\Omega)}dt
    \\+\int_0^\infty( y_n ,p(x)\eta )_{L^2(\Omega)}dt=- \int_0^\infty((\chi_\omega\psi_n) , \eta )_{L^2(\Omega)}dt.
\end{multline*}
We deduce that $( y_n)_t$ converge weakly to $\alpha_t$ in $L^2(0,\infty;H^{-1}(\Omega))$. Furthermore, from the uniqueness of the optimal variables of the problem \eqref{infinity_time_problem}, necessarily $ \hat{y}=\alpha$, $\hat{\psi}=\beta$, and $ f\wconv  \hat{f}$ in $L^2(0,\infty;H^1_0(\Omega))$.

Now, in order to obtain strong convergence, we use the structure of the functionals. Let us denote
\begin{align*}
    j_{n}=\inf_{f\in L^2(0,T_n;\Omega)}J^{T_n,0}(f), \quad\text{ and }\quad j=\inf_{f\in L^2(0,\infty;\Omega)}J^{\infty,0}(f).
\end{align*}
Since the functional is lower semicontinuous and convex, by the weak convergence, we have that 
\begin{align*}
    j\leq \liminf_{n\to\infty}j_{n}.
\end{align*}
On the other hand, since $ f_n$ and $ \hat{f}$ are the solutions of $j_n$ and $j$, respectively, then 
\begin{align*}
    j=J^{\infty,0}( \hat{f})=\limsup_{n\to\infty} J^{T_n,0}( \hat{f})\geq \limsup_{n\to\infty} J^{T_n,0}( f_n)= \limsup_{n\to\infty}j_n.
\end{align*}
Therefore, $j_n\to j$ as $n\to \infty$, ensuring the strong convergence of $ y_n$, $\psi_n$ and $ f_n$ to  $ \hat{y}$, $\hat{\psi}$ and $ \hat{f}$ in $L^2(0,\infty;\Omega)$, respectively.

{\bf Step 2: } From the strong convergence, we can pass the limit in the inequality \eqref{inequality_uniform_yn_psin} to obtain 
 \begin{align}\label{inequality_uniform_ystar_psistar}
    \int_0^{\infty}\| \hat{y}(\cdot,t)\|^2_{L^2(\Omega)}  +  \|\chi_\omega\hat{\psi}(\cdot,t)\|^2_{L^2(\Omega)}dt\leq C.
\end{align}
Now, recall the operator $\hat{E}(0): L^2(\Omega) \to L^2(\Omega)$ defined as  $
    \hat{E}(0)y_0(x) :=\hat{\psi}(x,0).
$
In the same way as the operator $\mathcal{E}$, we can check that $\hat{E}$ is well-defined, linear, and continuous. Also, if we restrict the problem \eqref{system_y*psi*} to the interval, $(s,\infty)$ we can define the operator $\hat{E}(s) \hat{y}(x,s) :=\hat{\psi}(x,s),$ for each $s\in(0,\infty)$. 

Since the coefficients of the equation \eqref{system_y*psi*} do not depend on the time variable, $\hat{E}$ is constant with respect to $s$. That is, $\hat{E}(s) \hat{y}(x,s)=\hat{E}(0) \hat{y}(x,s)$ for all $ \hat{y}(x,s)$ and $s\in (0,\infty)$.

Let us consider the sequences $( y_n,\psi_n)$ of optimal pairs of the functional $J^{T_n,0}$. Then, from the first part of the proof, we have that $\psi_n\to\hat{\psi}$. Therefore, using the definition of $\mathcal{E}$ and $\hat{E}$, we obtain
\begin{align*}
    \mathcal{E}(t)y_0\longrightarrow\hat{E}y_0\quad\text{strongly in }L^2(\Omega), \text{ when } t\to\infty,
\end{align*}
for any $y_0\in L^{2}(\Omega)$. Finally, since $ f_n(x,t)\to \hat{f}(x,t)$ and the characterization of the stationary optimal control, we obtain that $ \hat{f}(x,t)=-\chi_\omega\hat{E} \hat{y}(x,t)$.

\subsection{Proof of the Lemma \ref{corollary_1_exponencial_lemma}} \label{corollary_1_exponencial_lemma_proof}

First, observe that multiplying the system \eqref{system_y*psi*} by $ \hat{y}$ and integrating on $\Omega$, we have that 
\begin{align}\label{2.27}
    \frac{d}{dt}\biggr((\hat{E} \hat{y}(\cdot,t), \hat{y}(\cdot,t))_{L^2(\Omega)}\biggr) = -\|\chi_\omega \hat{E} \hat{y}(\cdot,t)\|_{L^2(\Omega)}^2-\| \hat{y}(\cdot,t)\|^2_{L^2(\Omega)}\leq0..
\end{align}
 Let us denote the operator $\mathcal{M}:=(\mathcal{A} +\chi_\omega \hat{E})$. We divide the proof into two steps.

{\bf Step 1: }We claim that 
\begin{align}\label{convergence_sup_lemma}
    \lim_{t\to\infty}\left( \sup_{\|y_0\|_{L^2(\Omega)}\leq 1  } (\hat{E} \hat{y}(\cdot,t), \hat{y}(\cdot,t))_{L^2(\Omega)}\right)=0,
\end{align}
Observe that from \eqref{2.27} the function defined by
\begin{align*}
    l(t):= \sup_{\|y_0\|_{L^2(\Omega)}\leq 1} (\hat{E} \hat{y}(\cdot,t), \hat{y}(\cdot,t))_{L^2(\Omega)},
\end{align*}
is non-increasing. Hence, it admits a limit as $t\to\infty$. To prove \eqref{convergence_sup_lemma}, observe that the operator $\mathcal{M}$ is a generator of a $C_0-$semigroup $S(t)$ (since $\hat{E}\in \mathcal{L}(L^2(\Omega))$), which depends of the coefficients $(a,b,p)$. 

Let us consider the sequence of times $t_n\to\infty$, and the sequence of initial condition $y_{0}^n$ such that the primal solution of \eqref{system_y*psi*} with initial condition $y_0^n$, given by $ \hat{y}_n(x,t)=S(t)y_{0}^n(x)$, satisfies
\begin{align*}
    l(t_n)-\frac{1}{n}\leq(\hat{E} \hat{y}_n(\cdot,t_n), \hat{y}_n(\cdot,t_n))_{L^2(\Omega)}.
\end{align*}
Note that the sequence $z_n(x,t)= \hat{y}_n(x,t+t_n)$ satisfies  
\begin{align}\label{equation_zn}
    (z_n)_{t} +M z_n = 0 \quad & (x,t)\in\Omega\times(-t_n,\infty),
\end{align}
which is the same equation satisfied by $ \hat{y}_n$. Then extending by zero $z_n$, in the interval $(-\infty,-t_n)$, we can apply Lemma \ref{lemma_1}, and the same arguments of the (first part) proof of the Lemma \ref{Lemma_3_exponencial}, to conclude that there exists $z$ such that $z_n\to z$ strongly in $L^2(\R;\Omega)$. Thus, $z$ satisfies the equation
\begin{align*}
    z_{t} +M z = 0, \quad & (x,t)\in\Omega\times\R.
\end{align*}
Observe that $t_n+t<t_n$ for every $t<0$. Then, using the definition of $l(t)$ and its monotone character, we have
\begin{align*}
    l(t_n)-\frac{1}{n}\leq(\hat{E} \hat{y}_n(\cdot,t_n), \hat{y}_n(\cdot,t_n))_{L^2(\Omega)}\leq(\hat{E} \hat{y}_n(\cdot,t_n+t), \hat{y}_n(\cdot,t_n+t))_{L^2(\Omega)}\leq l(t+t_n),
\end{align*}
for all $t<0$. This implies that $(\hat{E}z(\cdot,t),z(\cdot,t))_{L^2(\Omega)}$ is constant in $(-\infty,0)$. Moreover,
\begin{align*}
    (\hat{E}z(\cdot,t),z(\cdot,t))_{L^2(\Omega)}=\lim_{n\to\infty}l(t_n).
\end{align*}
Since $z$ and $ \hat{y}(\cdot,t)$ satisfy the same system, we apply \eqref{2.27} to $z$, obtaining
\begin{align*}
    \|\chi_\omega \hat{E}z(\cdot,t)\|_{L^2(\Omega)}^2=\|z(\cdot,t)\|^2_{L^2(\Omega)}=0,\quad \forall t\in (-\infty,0),\quad\forall\varepsilon>0,
\end{align*}
and thus $z(\cdot,t)=0$ for every $t\in (-\infty,0)$. Consequently,  $(\hat{E}z(\cdot,t),z(\cdot,t))_{L^2(\Omega)}=0$. Therefore,
$$\lim_{t\to\infty}l(t)=0,$$
 concluding \eqref{convergence_sup_lemma}.

{\bf Step 2:} Let us prove the uniform exponential decay of $ \hat{y}(x,t)$. To this end, we first prove that $\hat{E}$ is uniformly continuous. Note that, due to \eqref{convergence_sup_lemma}, we can integrate \eqref{2.27} in the time interval $(0,\infty)$. Obtaining the following variational characterization
\begin{align}\label{variational_charact_Ehat}
    (\hat{E}y_0(\cdot),y_0(\cdot))_{L^2(\Omega)} =\min_{ f\in L^{2}(0,\infty;\Omega)}J^{\infty,0}( f)\geq0.
\end{align}
The above implies that $\hat{E}$ is positive semi-definite. Since the uniform null controllability holds, the system can be steered to zero at any finite time. Then, we can use again the two-step control given by \eqref{two_step_control} (in this case with $T=\infty$). This guarantees that there exists some $C_{r}>0$, independent of the coefficients $(a,b,p)$, such that
\begin{align}\label{cota_problema_tiempoinfinito}
\min_{ f\in L^{2}(0,\infty;\Omega)}J^{\infty,0}( f)\leq C_{r}\|y_0(\cdot)\|^2_{L^2(\Omega)}.
\end{align}
On the other hand, if we multiply by $\hat{\psi}$ the equation satisfied by $\hat{\psi}$, integrating on $\Omega\times (0,\infty)$, integrating by parts and applying \eqref{variational_charact_Ehat}, we have
\begin{align*}
    \|\hat{\psi}(\cdot,0)\|_{L^2(\Omega)}^2+a_0 C_p\int_0^\infty\|\hat{\psi}(\cdot,t)\|_{L^2(\Omega)}^2dt\leq 
    \int_0^\infty\| \hat{y}(\cdot,t)\|_{L^2(\Omega)}^2dt.
\end{align*}
Then, using the previous estimation, the definition of $\hat{E}$ and \eqref{cota_problema_tiempoinfinito}, we conclude that
\begin{align}\label{uniform_continous_E}
    \|\hat{E}y_0\|^2_{L^2(\Omega)}\leq \int_0^\infty\| \hat{y}(\cdot,t)\|_{L^2}^2dt \leq \min_{ f\in L^{2}(0,\infty;\Omega)}J^{\infty,0}( f) \leq C_{r}\|y_0(\cdot)\|^2_{L^2(\Omega)},
\end{align}
Thus, $E$ is uniformly continuous. Also, we have 
\begin{align*}
     (\hat{E}y_0(\cdot),y_0(\cdot))_{L^2(\Omega)} &=\min_{ f\in L^{2}(0,\infty;\Omega)}J^{\infty,0}( f)\geq  \frac{1}{2}\int_0^\infty\| \hat{y}(\cdot,t)\|_{L^2(\Omega)}^2dt\geq  \frac{1}{2}\|\hat{\psi}(\cdot,0)\|_{L^2(\Omega)}^2= \frac{1}{2}\|\hat{E}y_0(\cdot)\|_{L^2(\Omega)}^2.
\end{align*}
Let us prove that there exist two positive constants $C$ and $\mu$, independent of the coefficients such that
\begin{align}\label{exponencial_estability_solution}
    \| \hat{y}(\cdot,t)\|_{L^2(\Omega)} \leq C e^{-\mu t} \|y_0(\cdot)\|_{L^2(\Omega)}.
\end{align}
From \eqref{uniform_continous_E}, we obtain
\begin{align}\label{energy_estimation_y}
  \| \hat{y}(\cdot,t)\|^2_{L^2(\Omega)}+\int_0^t\| \hat{y}(\cdot,s)\|^2_{L^2(\Omega)}ds \leq C_{r}\|y_0\|^2_{L^2(\Omega)} \quad\text{for every } t\geq 0.
\end{align}
Also we know that $ \hat{y}(\cdot,t)=S(t)y_0$, where $S(t)$ is the $C_0-$semigroup generated by $M$. Therefore, we have
\begin{align*}
    \|S(t)\|_{\mathcal{L}(L^2(\Omega))}^2\leq C_{r} \quad\text{ for every } t\geq 0.
\end{align*}
The above inequality implies that there exists $\nu>0$ independent of the coefficients $(a,b,p)$ and $T$ such that 
\begin{align*}
    \|S(t)\|_{\mathcal{L}(L^2(\Omega))}^2\leq C_{r} e^{2\nu t}\quad\text{ for every } t\geq 0.
\end{align*}
Furthermore, from \eqref{energy_estimation_y} we deduce that
\begin{align*}
    \int_0^t \|S(t)y_0\|^2_{L^{2}(\Omega)}dt \leq C_{r}\|y_0\|^2_{L^{2}(\Omega)}\quad\text{ for every }  t\geq 0.
\end{align*}
Consequently, we get
\begin{align*}
    t\|S(t)y_0\|^2_{L^{2}(\Omega)}&= \int_0^t\|S(t)y_0\|^2ds\leq \int_0^t\|S(s)\|^2_{\mathcal{L}(L^2(\Omega))}\|S(t-s)y_0\|^2_{L^{2}(\Omega)} ds \\
    &\leq C_{r}^2\int_0^t\|S(s)y_0\|^2ds\leq C_{r}^2\|y_0\|^2_{L^{2}(\Omega)},
\end{align*}
for all $t\geq0$. Then we have that  
\begin{align*}
    \|S(t)\|_{\mathcal{L}(L^2(\Omega))}\leq C_{r}\sqrt{\frac{1}{t}}\quad \text{for every } t>0.
\end{align*}
Hence, for $\hat{t}=C_{r}^2/\delta>0$ with $\delta\in(0,1)$, $$\|S(\hat{t})\|_{\mathcal{L}(L^2(\Omega))}\leq\delta<1.$$ On the other hand, we know that every $t>0$ can be written as $t=m\hat{t}+s$, where $m\in\N$ and $s\in (0,\hat{t}]$. Thus,
\begin{align*}
    \|S(t)\|_{\mathcal{L}(L^2(\Omega))}\leq C_{r} \|S(\hat{t})\|_{\mathcal{L}(L^2(\Omega))}^m e^{\nu s}\leq C_{r}\delta^{t/\hat{t}}\delta^{-1}e^{\nu\hat{t}}=C e^{-\mu t},
\end{align*}
and we conclude \eqref{exponencial_estability_solution} taking $$\mu=\frac{-\ln(\delta)\delta}{C_{r}^2}>0,\quad \text{and}\quad C=C_{r}\delta^{-1}e^{\nu\hat{t}},$$ which are independent of $T$ and the coefficients. Finally, denote by $S^*(t)$ the adjoint of the semigroup $S(t)$. Since $L^2(\Omega)$ is a Hilbert space, $S^*(t)$ is generated by $\mathcal{M}^*$. Furthermore,
$
    \|S^*(t)\|_{\mathcal{L}(L^2(\Omega))}=\|S(t)\|_{\mathcal{L}(L^2(\Omega))}.
$
This concludes the proof.

\subsection{Convergence of the optimal variables in the homogenization context}\label{homogenization_theorem_sec2_proof}

\begin{proposition}\label{homogenization_theorem_sec2}
Assume that $a\in L^{\infty}(\R)$ is a periodic function given by \eqref{homogenizated_a} satisfying \eqref{bound_a_homo} and $b\equiv p\equiv0$. Denote by  $(y,f)$ and $(\overline{y},\overline{f})$ the optimal pairs state-control of the optimization problems \eqref{evolutive_control_problem} and  \eqref{stationary_var_problem} subject to the equations \eqref{lim_system} and \eqref{lim_system_stationary}, respectively. Then, we have that
\begin{align*}
     f^\varepsilon\to f \quad \text{strongly in }  L^2(0,T;\Omega),\quad \text{and}\quad
     y^\varepsilon \to y \quad \text{strongly in }   C([0,T];L^2(\Omega)),
\end{align*}
as $\varepsilon\to 0$. Moreover, for the stationary optimal control variables, we have that
\begin{align*}
     \overline{f}^\varepsilon\to \overline{f},\quad\text{and}\quad \overline{y}^\varepsilon \to \overline{y},\quad\text{strongly in } L^2(\Omega),
\end{align*}
as $\varepsilon\to 0$.
\end{proposition}

We denote by $( f^\varepsilon,  y^\varepsilon)$ and $( \overline{f}^\varepsilon, \overline{y}^\varepsilon)$ the optimal variables associated to \eqref{evolutive_control_problem_2} and \eqref{stationary_var_problem_2} respectively. Also, denote by $y^h$ the solution of the homogenized parabolic equation \eqref{lim_system}, and $\overline{y}^h$ the solution of the homogenized elliptic equation \eqref{lim_system_stationary}.

We will make use of the following classical results of homogenization theory, whose proofs can be found in \cite{lions_besoussan,MR1172450}.

\begin{lemma}\label{Theorem_homogenization_systems_1}
Assume that $\Omega=(0,1)$, $a_\varepsilon\in L^{\infty}(\R)$ is a periodic function of period $1$ given by \eqref{homogenizated_a} satisfying \eqref{bound_a_homo}. Let $f^{\varepsilon}$ be a sequence of controls in $L^{2}(0, T;\Omega)$. Denote by $y^{\varepsilon}$ the solution of \eqref{rapidly_ocilation_general_heat} with right-hand side $f^{\varepsilon}$. Further, let $y^h$ be the solution of the homogenized parabolic equation \eqref{lim_system}. Then, the following properties hold:
\begin{enumerate}
    \item if $f^{\varepsilon}$ converges to $f^h$ weakly in $ L^{2}(0, T;\Omega)$ as $\varepsilon \rightarrow 0$, then $y^{\varepsilon}$ satisfies
\begin{align*}
y^{\varepsilon} \wconvs y^h \text{ weakly-* in } L^{\infty}(0, T ; L^{2}(\Omega)),
\end{align*}  
as $\varepsilon \rightarrow 0$.
 \item if $f^{\varepsilon}$ converges to $f^h$ strongly in $ L^{2}(0, T;\Omega)$ as $\varepsilon \rightarrow 0$, then $y^{\varepsilon}$ satisfies
 \begin{align*}
     y^{\varepsilon} \rightarrow y^h \text{ strongly in } C([0, T] ; L^{2}(\Omega)),
\end{align*} 
as $\varepsilon \rightarrow 0$.
\end{enumerate}
\end{lemma}

\begin{lemma}\label{Theorem_homogenization_systems_2}
    Assume that $\Omega=(0,1)$, $a_\varepsilon\in L^{\infty}(\R)$ is a periodic function of period $1$ given by \eqref{homogenizated_a} satisfying \eqref{bound_a_homo}. Let $ \overline{f}^\varepsilon$ a sequence of controls in $L^2(\Omega)$ such that converge weakly in $L^2(\Omega)$ to $\overline{f}^h$ as $\varepsilon \rightarrow 0$. Then, $\overline{y}^\varepsilon$ the solution of \eqref{stationry_system} with right hand side $ \overline{f}^\varepsilon$ satisfies
\begin{align*}
        \overline{y}^\varepsilon \wconv \overline{y}^h,\quad \text{weakly in }  H_0^1(\Omega),
\end{align*}
as $\varepsilon \rightarrow 0$.
\end{lemma}

Let us start by observing that, from the uniform convexity of $J^T$, we have that $f^\varepsilon$ and $ y^\varepsilon$ are uniformly bounded in $L^{2}(0, T;\Omega)$. Then there exist $f_l$ and $y_l$ in $L^{2}(0, T;\Omega)$ such that  
\begin{align*}
     f^\varepsilon\wconv f_l,\quad\text{and}\quad y^\varepsilon\wconv y_l\quad\text{ weakly in }L^{2}(0, T;\Omega),
\end{align*}
as $\varepsilon\to 0$. In particular, using Lemma \ref{Theorem_homogenization_systems_1}, 
\begin{align*}
y^{\varepsilon} \wconvs y_l \text{ weakly-* in } L^{\infty}(0, T ; L^{2}(\Omega))
\text{ as }\varepsilon \rightarrow 0.
\end{align*}    
Moreover,  $y_l$ is the solution of the limit system \eqref{lim_system} associated with $f_l$. 
Let us prove that $y_l=y^h$ and $f_l=f^h$. From the lower semicontinuity of the functional $J^T$, we have
\begin{align}\label{des1}
    J^{T}(f_l)\leq \liminf_{\varepsilon\to0} J^{T}( f^\varepsilon).
\end{align}
On the other hand, for every $g\in L^{2}(0, T;\Omega)$, we deduce
\begin{align}\label{des2}
    \liminf_{\varepsilon\to0} J^{T}( f^\varepsilon)\leq \liminf_{\varepsilon\to0} J^{T}(y^{\varepsilon}(h),h),
\end{align}
where $y^{\varepsilon}(g)$ is the state associate to $g$. Then, since $g$ is a constant control, using Lemma \ref{Theorem_homogenization_systems_1}, we obtain $y^{\varepsilon}(g) \rightarrow y^h(g)$ strongly in $C([0, T] ; L^{2}(\Omega))$ as $\varepsilon\to 0$. Thus,
\begin{align}\label{des3}
    \liminf_{\varepsilon\to0} J^{T}(y^{\varepsilon}(g),g)= J^{T}(y^h(g),g)=J^{T}(g)
\end{align}
Since \eqref{des1}, \eqref{des2} and \eqref{des3} hold, we obtain
\begin{align*}
      J^{T}(f_l)\leq J^{T}(g)\quad\text{for every }g\in L^{2}(0, T;\Omega).
\end{align*}

Then, we conclude that $f_l=f^h$ and by the uniqueness, $y_l=y^h$. Now, to conclude the strong convergence, using the weak convergence and the weak semicontinuity of $J^T$, we have
\begin{align*}
   \liminf_{\varepsilon\to 0} J^{T}(y^\varepsilon( f^\varepsilon), f^\varepsilon)\geq  J^T(y^h(f^h),f^h)\geq \limsup_{\varepsilon\to0} J^{T}(y^\varepsilon(f^h),f^h)\geq\limsup_{\varepsilon\to 0}J^{T}(y^\varepsilon(f^\varepsilon), f^\varepsilon).
\end{align*}
Thus, we deduce
$$\lim_{\varepsilon\to 0}J^{T}( f^\varepsilon)=J^{T}(f^h),$$ and we conclude the strong convergence.

Finally, let us prove that $\overline{y} \to \overline{y}$ and $ \overline{f}\to \overline{f}$ strongly in $L^2(\Omega)$. From the strict convexity of the functional $J^s_{\varepsilon}$, we have that there exist $f_s$, $y_s\in L^{2}(\Omega)$ such that \begin{align*}
     \overline{f}^\varepsilon\wconv f_s,\quad \text{ and } \quad \overline{y}^\varepsilon\wconv y_s,\quad  \text{weakly in } L^{2}(\Omega).
\end{align*} 
From Lemma \ref{Theorem_homogenization_systems_1}, we know that $y_s$ is solution of \eqref{lim_system_stationary}. In the same way as above, combining the weak convergence, we can derive \eqref{des1}, \eqref{des2}, and \eqref{des3} for the functional $J^{s}$. Therefore, we obtain that $y_s=\overline{y}^h$ and $f_{s}=\overline{f}^h$. Furthermore, we have
\begin{align*}
    \lim_{\varepsilon\to 0}J^{s}( \overline{f}^\varepsilon) = J^{s}(\overline{f}^h).
\end{align*}
This proves the strong convergence in $L^2(\Omega)$.


\vspace{5mm}
\section*{Acknowledgments}

M. Hernandez has been funded by the Transregio 154 Project, “Mathematical Modelling, Simulation, and Optimization Using the Example of Gas Networks” of the DFG, project C07, and the fellowship "ANID-DAAD bilateral agreement" 57600326.

E. Zuazua has been funded by the Alexander von Humboldt-Professorship program, the ModConFlex Marie Curie Action, HORIZON-MSCA-2021-DN-01, the COST Action MAT-DYN-NET, the Transregio 154 Project “Mathematical Modelling, Simulation and Optimization Using the Example of Gas Networks” of the DFG, grants PID2020-112617GB-C22 and TED2021-131390B-I00 of MINECO (Spain).
Madrid Government - UAM Agreement for the Excellence of the University Research Staff in the context of the V PRICIT (Regional Programme of Research and Technological Innovation).\\

\vspace{5mm}

\bibliographystyle{abbrv} 
\bibliography{biblio_ut.bib}
\vfill

\end{document}